\documentclass[a4paper,fleqn]{cas-sc}
\usepackage{enumitem}
\usepackage{graphicx} 
\usepackage{mathtools}
\usepackage[authoryear,longnamesfirst]{natbib}
\setcitestyle{authoryear,open={[},close={]},aysep={,}}

\usepackage{caption}
\usepackage{floatrow}
\usepackage{xurl} 
\AtEndPreamble{\hypersetup{hidelinks,
breaklinks=true,
colorlinks=true,
linkcolor=blue,
citecolor=blue,
urlcolor=blue,
bookmarksopen=false,
pdftitle={Title},
pdfauthor={Author}}}

\newtheorem{theorem}{Theorem} 
\newtheorem{remark}[theorem]{Remark}
\newtheorem{lemma}[theorem]{Lemma}
\newtheorem{proposition}[theorem]{Proposition}
\newtheorem{corollary}[theorem]{Corollary}
\newtheorem{example}[theorem]{Example}
\newproof{proof}{Proof}

\newcommand{\zmerspthing}{z_{\mbox{\normalfont\protect\tiny MERSP}}}
\newcommand{\zmersp}{\hyperlink{zmersptarget}{\zmerspthing}}

\newcommand{\zmespthing}{z_{\mbox{\normalfont\protect\tiny MESP}}}

\newcommand{\IdentityNLP}{\hyperlink{Identitytarget}{\mbox{Identity}}}
\newcommand{\DiagonalNLP}{\hyperlink{Identitytarget}{\mbox{Diagonal}}}
\newcommand{\TraceNLP}{\hyperlink{Identitytarget}{\mbox{Trace}}}

\newcommand{\Rfix}{\hyperlink{Rfixtarget}{\mathcal{R}}}

\newcommand{\Sfix}{\hyperlink{Sfixtarget}{\mathcal{C}}}

\newcommand{\znlpthing}{{z}_{\text{\normalfont\tiny hNLP}}}
\newcommand{\znlp}{\hyperlink{znlptarget}{\znlpthing}}

\newcommand{\NLPcompMERSP}{$\overline{\vphantom{t}\smash{\mbox{hNLP}}}$$(\mbox{MERSP})$}

\newcommand{\NLPorigMERSP}
{{\rm hNLP}$(\mbox{MERSP})$}

\newcommand{\zspectralthing}{z_{\mbox{\normalfont\protect\tiny $\mathcal{S}$}}}
\newcommand{\zspectral}{\hyperlink{zspectraltarget}{\zspectralthing}}

\DeclareMathOperator{\Diag}{Diag}
\DeclareMathOperator{\diag}{diag}
\DeclareMathOperator{\ldet}{ldet}

\DeclareMathOperator{\Tr}{tr}
\DeclareMathOperator{\rank}{rank}
\DeclareMathOperator{\argmin}{argmin}

\makeatletter
\renewcommand*{\top}{%
  {\mathpalette\@transpose{}}%
}
\newcommand*{\@transpose}[2]{%
  \raisebox{\depth}{$\m@th#1\scriptscriptstyle\mathsf{T}$}%
}
\makeatother

\begin{document}

\let\WriteBookmarks\relax
\def\floatpagepagefraction{1}
\def\textpagefraction{.001}
\shorttitle{The hyper-scaled NLP bound for MERSP}
\shortauthors{G.~Ponte, M.~Fampa, J.~Lee}

\title[mode = title]{The hyper-scaled NLP bound for\texorpdfstring{\\}{} maximum-entropy remote sampling} 

\author[1]{Gabriel Ponte}[orcid=0000-0002-8878-6647]
\ead{gabponte@umich.edu}
\ead[URL]{https://sites.google.com/umich.edu/gabrielponte/}

\author[2]{Marcia Fampa}[orcid=0000-0002-6254-1510]
\ead{fampa@cos.ufrj.br}
\ead[URL]{https://marciafampa.com/}

\author[3]{Jon Lee}[orcid=0000-0002-8190-1091]
\ead{jonxlee@umich.edu}
\ead[URL]{https://sites.google.com/site/jonleewebpage/}

\affiliation[1]{organization={University of Michigan},city={Ann Arbor},state={Michigan},country={U.S.A.},}
\affiliation[2]{organization={Universidade Federal do Rio de Janeiro},country={Brasil}}
\affiliation[3]{organization={University of Michigan},city={Ann Arbor},state={Michigan},country={U.S.A.}}

\begin{abstract}
The maximum-entropy remote sampling problem (MERSP) is
to select a subset of $s$ random variables from a set of $n$
random variables, so as to maximize the information 
concerning a set of target random variables that are not directly observable. We assume throughout that the set
of all of these random variables follows a joint Gaussian distribution, and that we have the covariance matrix available. Finally, we measure information using
Shannon's differential entropy.

The main approach for exact solution of moderate-sized instances of MERSP has been 
branch-and-bound (B\&B), and so previous work concentrated on
upper bounds. Prior to our work, there were two 
upper-bounding methods for MERSP: the so-called ``complementary NLP bound'' and the ``spectral bound'', both introduced 25 years ago.
We are able now to establish domination results between these two upper bounds. 
Further, we propose a novel and effective ``hyper-scaled NLP bound'' (hNLP  bound) based on 
a subtle convex relaxation. 
The ``complementary'' version of hNLP bound for MERSP generalizes the previous complementary NLP bound for MERSP. 
We provide theoretical guarantees, giving sufficient conditions under which the complementary hNLP bound strictly dominates the complementary NLP bound. In addition, the hNLP formulation allows us to derive upper bounds for rank-deficient covariance matrices when they satisfy a technical condition. This is in contrast to the previous NLP bound that worked 
with only positive definite covariance matrices (because it was wedded to a complementary formulation). Additionally, we describe procedures for calculating hyper-scaling parameters. Finally, for B\&B, we provide a variable-fixing methodology and results guiding the best way to construct subproblems.
Numerical experiments on benchmark instances demonstrate the effectiveness of our approaches in advancing the algorithmic state-of-the-art for MERSP. 
\end{abstract}


\maketitle


\section{Introduction}
\label{sec:intro}

Let $Y$ be a random vector, indexed by $N\cup T$, where (without loss of generality)
$N:=\{1,\ldots,n\}$,  $|T|=t$, and
$N\cap T=\emptyset$. 
Given an integer $s$ satisfying $0<s<n$,
our goal is to choose a subvector $Y_S$ of $Y_N$\,, with 
$|S|=s$, so as to minimize the remaining information in $Y_T$\,, after conditioning
on $Y_S$\,. The idea is that we cannot directly observe the ``target''
random vector $Y_T$\,, and we want to gain as much information as
possible about it, by observing only a subvector of size $s$ of the
observable random vector $Y_N$\,. The restriction on the cardinality
of $S$ is motivated by sampling being costly. 

As is usual, we measure information 
using Shannon's ``differential entropy'' (see \citep*{Shannon}), and 
we assume that $Y$ has a joint Gaussian distribution,
with covariance matrix denoted by $C:=C[N\cup T,N\cup T]$. 
In this setup, it turns out (see \citep*{AFLW_Remote}) that we can view our
problem as seeking to minimize the natural logarithm of the determinant of
the conditional covariance matrix
$
C_S[T,T]:=C[T,T]-C[T,S]C[S,S]^{-1}C[S,T],
$
over $S\subset N$ with $|S|=s$, where we use 
$C[P,Q]$ to denote the submatrix of $C$  with row 
indices $P$ and column indices $Q$.
With the mild assumption that $C[T,T]\succ 0$,
as was pointed out in \citep*{AFLW_Remote}, it is more convenient and equivalent to view this as maximizing 
 $\ldet (C[S,S]) - \ldet(C_T[S,S])$,  over $S\subset N$ with $|S|=s$, where $\ldet(\cdot)$ denotes the natural logarithm of the determinant, 
 and this specific formulation is known as the  \emph{maximum-entropy remote sampling problem} (MERSP); see \citep*{AFLW_Remote}. 
 We note that without the term ``$- \ldet (C_T[S,S])$'', the
 problem is known as the \emph{maximum-entropy sampling problem} (MESP),
 which has been the subject of intense study until these days (see \citep*{FLbook} and the many references therein, and even more recently \citep*{FL_update}, \citep*{ADMM4DOPT}, \citep*{MESP2DOPT}). Motivated by the many algorithmic techniques
 and results that have been developed and are continuing to be
 developed for MESP, we seek to attack the more difficult MERSP,
 a problem that has been largely neglected for the last 25 years.
For convenience
throughout, we write $B:=C[N,N]$ and $B_T :=C_T[N,N]$. 
Letting $S(x)$ denote the support of $x\in\{0,1\}^n$,  and $\mathbf{e}$ denote an all-ones vector, 
we formally define
 the \emph{maximum-entropy remote sampling problem} as
 \begin{align*}\tag{MERSP}\label{MERSP}
&\zmerspthing(B,B_T\,,s):=\max \left\{\ldet (B[S(x),S(x)]) - \ldet (B_T[S(x),S(x)])~:~
\mathbf{e}^\top x =s,~ x\in\{0,1\}^n \right\},
\end{align*} 
and the closely related \emph{maximum-entropy sampling problem} as
\begin{equation*}\tag{MESP}\label{MESP}
\zmespthing(B,s):= \max \left\{\ldet (B[S(x),S(x)]) ~:~\mathbf{e}^\top x =s,~  x\in\{0,1\}^n\right\},
\end{equation*}
where $T=\emptyset$ in the latter. To formally see that \ref{MESP} is a special case of 
\ref{MERSP}, we refer to \citep*[Section 2]{AFLW_Remote}.

\medskip

\noindent {\bf Literature.} \ref{MESP} was introduced in \citep*{SW}, and subsequently proven to be NP-hard and approached for global optimization by branch-and-bound (B\&B) in \citep*{KLQ}.
A rather comprehensive work on  
\ref{MESP} is the recent monograph  \citep*{FLbook}, with very recent advances in
\citep*{FL_update}, \citep*{ADMM4DOPT}, \citep*{MESP2DOPT}.
Upper bounds from a variety of 
convex relaxations of \ref{MESP} are called 
the ``NLP bound''  
(see \citep*{AFLW_Using}, \citep*{AFLW_IPCO}),
the ``BQP bound'' (see \citep*{Anstreicher_BQP_entropy}),
the ``linx bound'' (see \citep*{Kurt_linx}),
and the ``factorization bound''
(see \citep*{Nikolov}, \citep*{Weijun}, \citep*{FactPaper}, \citep*{li2025augmented}).
Many of these bounds can take advantage of a scaling technique,
first introduced in \citep*{AFLW_Using}, \citep*{AFLW_IPCO}),
and later generalized to ``g-scaling'' (see \citep*{ACDA2023}, \citep*{gscale}).
\ref{MERSP} was introduced by 
\citep*{Bueso} and \citep*{AFLW_Remote};
also see \citep*{LeeDiscuss}.
\citep*{AFLW_Remote} demonstrated its NP-hardness, 
extended a  ``spectral bound''
for constrained \ref{MESP} (see \citep*{LeeConstrained}) to
a spectral bound for constrained \ref{MERSP},
and extended the convex-programming ``NLP bound'' from \ref{MESP}
to a complementary formulation of \ref{MERSP}. In what follows, when referring to the NLP bound for \ref{MESP} in particular, we state this explicitly so as to avoid confusion. 

\medskip
\noindent {\bf Our approach.} 
Due to the continuous
concavity of the $\ldet$
function, it has been natural 
to seek 
upper bounds for
the optimal objective 
value of
\ref{MESP} based on
convex-programming 
relaxations;
indeed, this approach
has been very successful
(see \citep*{FLbook}, and the
references 
above).
But \ref{MERSP} is significantly 
more challenging,
because its (maximization) 
objective
function has also a negative 
$\ldet$ term. 

In this work, we leverage a few key techniques for obtaining
upper-bounds for \ref{MERSP}.  \citep*{AFLW_Remote} pointed out (also see \citep*[Section 1.6]{FLbook}) an ordinary ``scaling'' (o-scaling) principle for \ref{MERSP}:
that is, for $\gamma > 0$, we have $ \zmerspthing(B,B_T\,,s)= \zmerspthing(\gamma B,\gamma B_T\,,s)$. Additionally, \citep*{gscale} developed  
``generalized scaling'' (g-scaling) principles, which inspire our work.
Finally, \citep*{AFLW_Remote} pointed out (also see \citep*[Section 1.6]{FLbook}) a ``complementing'' principle for \ref{MERSP}:
If $C$ is positive definite, 
then  
\begin{equation}\label{eq:complementation}
    \zmerspthing(B,B_T\,,s)=\zmerspthing(B^{-1},B_T^{-1},n-s) + \ldet (B) - \ldet(B_T).
\end{equation} 
We note that any upper bound for $\zmersp(B^{-1},B_T^{-1},n-s)$ plus $\ldet (B)-\ldet(B_T)$ is also an upper bound for $\zmersp(B,B_T\,,s)$, which we call a \emph{complementary upper bound}. 
These principles, ordinary scaling (o-scaling),
generalized scaling (g-scaling), and complementing, are referred to in what follows.

\medskip
\noindent {\bf Organization and Contributions.} 
In \S\ref{sec:hyper-nlp-bound}, we propose a novel  ``hyper-scaled NLP bound'' (hNLP bound) based on 
a subtle convex relaxation
of a generalization of \ref{MERSP}, the ``maximum-entropy difference problem'' (\ref{MEDP}),
a problem also considered by \citep*{AFLW_Remote}.
The hNLP bound has: 
(i) an o-scaling parameter
$\gamma\in\mathbb{R}_{++}$ (this is 
the o-scaling parameter of the NLP bound for \ref{MEDP} proposed in \citep*{AFLW_Remote}), 
(ii) a further  scaling parameter $\psi\in\mathbb{R}_{++}$ (which has no precedent in the literature), and 
(iii) a g-scaling parameter 
$\Phi\in\mathbb{R}^n_{++}$ 
(closely related to the g-scaling parameters in 
\citep*{gscale}).
The complementary hNLP bound for \ref{MERSP} generalizes the complementary NLP bound for \ref{MERSP}.
We also review the three parameter-selection strategies for the NLP bound for \ref{MEDP} proposed in \citep*{AFLW_Remote}, namely Identity, Diagonal and Trace, which we employ as well.
In \S\ref{sec:compare-mersp-bounds}, 
we demonstrate that under certain technical sufficient conditions, the complementary NLP bound for \ref{MERSP}
dominates the spectral bound for \ref{MERSP},
bounds both proposed in \citep*{AFLW_Remote}. 
Because the complementary hNLP bound for \ref{MERSP} trivially dominates the  complementary NLP bound
for \ref{MERSP},
under the same technical conditions, 
 the complementary hNLP bound for \ref{MERSP} dominates the spectral bound
for \ref{MERSP}. 
In \S\ref{sec:aug-nlp-bound}, we study the 
computation of the parameter $\psi$. 
First, we observe that while the  complementary NLP bound for \ref{MERSP} 
of \citep*{AFLW_Remote} is tied to the complementary formulation of \ref{MERSP}, we have that under some mild technical conditions, our hNLP bound 
applies also to the original formulation of \ref{MERSP}. 
Because of this, it applies even when either $B$ or $B_T$ is singular. We study this in detail. Then we
describe how to optimally choose the parameter $\psi$, via closed form expressions. 
In \S\ref{sec:diag-scale}, we propose a practical strategy for selecting the g-scaling parameter $\Phi$, and we show how to optimize it for the hNLP bound with the Identity strategy. 
In \S\ref{sec:BB},  for B\&B, we provide a variable-fixing methodology and results guiding the best way to construct subproblems.
In \S\ref{sec:exp}, we present results from numerical experiments that demonstrate the benefits of the  hNLP bound 
and the impact of using the scaling parameters $\psi$ and $\Phi$, both individually and in combination. In particular, we exhibit significant improvements on real-data instances compared to the approach in \citep*{AFLW_Remote}.
In \S\ref{sec:out}, we describe some directions for further study.

\medskip
\noindent{\bf Further Notation.}
Throughout, we denote any all-zero 
matrix simply by $0$, and we denote any zero (column) vector by $\mathbf{0}$.
We denote
any \hbox{$i$-th} standard unit vector by $\mathbf{e}_i$\,,  
 and the order-$n$ identity matrix by $I_n$\,.
 We let $\mathbb{S}^n$  (resp., $\mathbb{S}^n_+$~, $\mathbb{S}^n_{++}$)
 denote the set of order-$n$ symmetric (resp., positive-semidefinite, positive-definite) matrices.
 We let 
 \begin{itemize}
 \item $\Diag(x)$ denote the $n\times n$ diagonal matrix with diagonal elements given by the components of $x\in \mathbb{R}^n$;
 \item $\diag(X)$ denote the $n$-vector with elements given by the diagonal elements of $X\in\mathbb{R}^{n\times n}$;
\item  $\Diag(X):=\Diag(\diag(X))\in\mathbb{R}^{n\times n}$, for $X\in\mathbb{R}^{n\times n}$.
\end{itemize}
   For a matrix $X$ and index sets $P,Q$, $X[P,Q]$ denotes the submatrix  with row (column) indices $P$ ($Q$).
We use
$\Tr(\cdot)$ the trace,
and we use $(\cdot)^\dagger$ for the Moore-Penrose pseudoinverse.
For $X\in\mathbb{S}^n$,
we let $\lambda_i(X)$ 
denote its $i$-th greatest eigenvalue. 
We let  $\lambda(X):=(\lambda_1(X),\lambda_2(X),\ldots,\lambda_n(X))^\top$.
For $X\in\mathbb{R}^{m\times n}$,
we let  $\sigma_i(X)$
denote its $i$-th greatest singular value. 
We let  $\sigma(X):=(\sigma_1(X),\sigma_2(X),\ldots,\sigma_{\min\{m,n\}}(X))^\top$.
For a matrix $X := 
   \left(\begin{smallmatrix}
       U & W\\
       Y & Z
   \end{smallmatrix}\right)$, 
    we denote the Schur complement of the block $U$ in  $X$ by $X/U$.
For a vector $d\in\mathbb{R}^n$,
which plays a special role for us, we let $d_{\max}$
(resp.,  $d_{\min}$)
denote its maximum (resp., minimum) component. For $x\in \mathbb{R}^n_+$\,, $p\in\mathbb{R}^n_{++}$\,, we write $x^p$ to denote the vector in $\mathbb{R}^n_+$ having $(x^p)_i:=(x_i)^{p_i}$. 


\section{The hyper-scaled  NLP bound}\label{sec:hyper-nlp-bound}
Let $C_1\,, C_2 \in \mathbb{S}^n_{+}$\,,   
and let $s$ be an integer satisfying
$0 < s<  n$. Consider the following more-general problem than \ref{MERSP}, the \emph{maximum-entropy difference problem}: 
\begin{equation}\label{MEDP}\tag{MEDP}
z(C_1,C_2,s):= \max\left\{\ldet (C_1[S(x),S(x)])  - \ldet (C_2[S(x),S(x)]) \,:\, \mathbf{e}^\top x =s,\,x\in\{0,1\}^n\right\}.
\end{equation}
Notice that \ref{MEDP} particularizes to \ref{MERSP} by setting $C_1:=B$ and $C_2:=B_T$\,.

Consider given parameters $d,p \in \mathbb{R}^n_{++}$\,, and $\gamma > 0$. For $\psi > 0$, $\Phi\in\mathbb{R}^n_{++}$\,, we define
\[
f_1(x;\psi,\Phi) := \ldet\big(\Diag((\gamma d)^x) +  
\gamma\Diag(x^{p/2})(\psi \Diag(\Phi)C_1\Diag(\Phi)-\Diag(d))\Diag(x^{p/2}) \big) - s\log \psi,
\]
\[
f_2(x;\Phi) := \ldet\big(\Diag((\gamma d)^x) +  
\gamma\Diag(x^{p/2})(\Diag(\Phi)C_2\Diag(\Phi)-\Diag(d))\Diag(x^{p/2}) \big),
\]
and $f(x;\psi,\Phi) := f_1(x;\psi,\Phi) - f_2(x;\Phi)$. 
Then, the \emph{hyper-scaled NLP bound} for \ref{MEDP}, 
 is given by 
    \begin{align}\label{hNLP}\tag{hNLP}
\textstyle
\hypertarget{znlptarget}{\znlpthing}(C_1\,,C_2\,,s\,;\psi,\Phi) := \max \{&f(x;\psi,\Phi) \, :\, \mathbf{e}^\top x\!=\!s,\,   
x\in[0,1]^n\}.\nonumber 
\end{align}

It is important to notice that the parameter $\psi$ appears in $f_1$ but not in $f_2$\,.

The following proposition establishes that, for any $\psi>0$ and $\Phi\in\mathbb{R}^n_{++}$\,, \ref{hNLP} is a relaxation of \ref{MEDP}.

\begin{proposition}\label{prop:exact_rel}
Let $\psi>0$, $\Phi\in\mathbb{R}^n_{++}$\,, and $\hat{x}\in\{0,1\}^n$ such that $\mathbf{e}^\top \hat{x}=s$. Then, 
\[
f(\hat{x};\psi,\Phi)= \ldet ( C_1[S(\hat{x}),S(\hat{x})] ) - \ldet ( C_2[S(\hat{x}),S(\hat{x})] ).
\]
\end{proposition}

\begin{proof} 
  Note that because $\Diag(\Phi)$
  is a diagonal matrix with positive diagonal entries, we have that
 \[
 \left(\Diag(\Phi)C_k\Diag(\Phi)\right)[S(\hat{x}),S(\hat{x})] = (\Diag(\Phi))[S(\hat{x}),S(\hat{x})]\,C_k[S(\hat{x}),S(\hat{x})]\,(\Diag(\Phi))[S(\hat{x}),S(\hat{x})],
 \]
 for $k=1,2$. Then, it is  straightforward to verify that 
\begin{align*}
&f_1(\hat{x};\psi,\Phi)= \ldet \left((\psi\gamma\Diag(\Phi)C_1\Diag(\Phi))[S(\hat{x}),S(\hat{x})]\right) - s\log \psi\\
&\qquad =\ldet ( C_1[S(\hat{x}),S(\hat{x})] )+ s\log \gamma + 2\ldet(\Diag(\Phi)),\\
&f_2(\hat{x};\Phi)= \ldet \left((\gamma\Diag(\Phi)C_2\Diag(\Phi))[S(\hat{x}),S(\hat{x})]\right)\\
&\qquad = \ldet ( C_2[S(\hat{x}),S(\hat{x})] ) + s\log \gamma + 2\ldet(\Diag(\Phi)).
\end{align*}
The result follows.
\qed\end{proof}

Leveraging results
from \citep*{AFLW_Using}
on parameter selection, 
\citep*{AFLW_Remote}  analyzed the best choices of the parameters $d$, $p$, and $\gamma$ for the fundamental case of $\psi:=1$, $\Phi:=\mathbf{e}$, which assures convexity of \ref{hNLP} in this case. The strategies are based in the following  two results.

\begin{theorem}[\rm\protect{\citep*{AFLW_Remote}, Theorem 4.1}]\label{thm:nlpbound-mersp}
    Assume that $ \Diag(d) \succeq  C_2 \succeq  C_1$\,, $p \geq \mathbf{e}$, $\gamma > 0$, and $0<\gamma d_i\leq \exp(p_i - \sqrt{p_i})$ for $i \in N$. Then $f(\cdot\,;1,\mathbf{e})$ is concave on $0 < x \leq \mathbf{e}$. 
\end{theorem}

\begin{theorem}[\rm\protect{\citep*{AFLW_Remote}, Equation 15}]\label{cor:best-p-nlp}
     Assume that  $C_1$\,, $C_2$\,, $\gamma$, $p$ and $d$ satisfy the conditions of Theorem \ref{thm:nlpbound-mersp}. The least value of ${\znlpthing}(C_1\,,C_2\,,s\,;1,\mathbf{e})$
     is obtained using,     for $i \in N$,
    \begin{align*}
        p_i:=\begin{cases}
        1, &    \text{for }\gamma d_i\leq 1;
        \\
        \left(1+\sqrt{1+4\log(\gamma d_i)}\right)^2/4,\quad &\text{for }\gamma d_i> 1.
    \end{cases}
    \end{align*}
\end{theorem}

For the case of $\psi:=1$, $\Phi:=\mathbf{e}$, \citep*{AFLW_Remote} proposed  three strategies for selecting the parameters $d$ and $\gamma$ in \ref{hNLP}, aiming at obtaining a good bound for  
\ref{MEDP} while assuring convexity of \ref{hNLP}. Once these parameters are chosen,  $p$ is determined using Theorem \ref{cor:best-p-nlp}. 

We note that these ideas were originally introduced for the ``NLP bound'' (in either its original or complementary form) for \ref{MESP}; see \citep*{AFLW_Using}. Considering $D:=\Diag(d)$, the three strategies are:

\begin{itemize}
    \item \hypertarget{Identitytarget}{\IdentityNLP}:
$
        D:= \rho I_n\,,~ \rho:=\lambda_{1}(C_2)\,,~ \gamma:= 1/\rho.
$
By Theorem \ref{cor:best-p-nlp}, we have then that $p=
\mathbf{e}$.
    \item \hypertarget{Diagonaltarget}{\DiagonalNLP}: 
$
        D:= \rho\Diag(C_2),~ \rho:=\lambda_{1}(\Diag(C_2)^{ \scriptscriptstyle -1/2}C_2\Diag(C_2)^{\scriptscriptstyle -1/2}),~ \gamma\in  [1/d_{\max}\,,1/d_{\min}].
$
    \item \label{eq:nlp-trace} 
\phantom{.}
\vspace{-21pt}
    \begin{equation}
\phantom{.}\hspace{-27.5pt}\hypertarget{Tracetarget}{\TraceNLP}\!\!:        D:= \argmin_Y\{\Tr(Y)\,:\,Y-C_2\succeq 0,\,Y \text{ diagonal}\},~ \gamma\in  [1/d_{\max}\,,1/d_{\min}].
    \end{equation}
\end{itemize}

From Theorem \ref{thm:nlpbound-mersp}, we  see that when $\psi=1$ and $\Phi=\mathbf{e}$, a suitable choice of the parameters $\gamma$, $d$, and $p$ guarantees the convexity of \ref{hNLP}. Next, building on Proposition \ref{prop:exact_rel} and Theorem \ref{thm:nlpbound-mersp}, we demonstrate  that for \emph{any} $\Phi\in\mathbb{R}^n_{++}$\,, and suitable choices of $\psi$, $\gamma$, $d$, and $p$, \ref{hNLP} constitutes a convex relaxation of \ref{MEDP}.

\begin{corollary}\label{cor:hypernlpbound-mersp}
    Let $\psi>0$ and $\Phi\in\mathbb{R}^n_{++}$\,. Assume that $C_2 \succeq  \psi C_1$\,. If         $\Diag(d) \succeq  \Diag(\Phi)C_2\Diag(\Phi)$\,, $p \geq \mathbf{e}$, $\gamma > 0$, and $0<\gamma d_i\leq \exp(p_i - \sqrt{p_i})$ for $i \in N$, then $f(\cdot\,;\psi,\Phi)$ is concave on $0 < x \leq \mathbf{e}$. 
\end{corollary}

\begin{proof}
Notice that if  $C_2\succeq \psi {C}_1$ then  $\Diag(\Phi) C_2 \Diag(\Phi)\succeq \psi \Diag(\Phi) {C}_1 \Diag(\Phi)$. Then, the result immediately follows from Theorem \ref{thm:nlpbound-mersp}.
\qed
\end{proof}

Finally, we can 
extend any of the $\IdentityNLP$, $\DiagonalNLP$, or $\TraceNLP$  strategies described above for selecting the parameters $d$ and $\gamma$ in  \ref{hNLP} when $\psi=1$ and $\Phi=\mathbf{e}$, to the more general case where $\psi>0$ and $\Phi\in\mathbb{R}^n_{++}$\,. For the computation of the parameters, we simply  replace $C_2$ by $\Diag(\Phi) C_2 \Diag(\Phi)$. We define the relaxations obtained this way, by hNLP-Id, hNLP-Di, and hNLP-Tr (selecting $p$ using Theorem \ref{cor:best-p-nlp}).
It is important to note that for given matrices $C_1$ and $C_2$\,, and parameter $\psi>0$, if $C_2\succeq \psi C_1$\,, then all assumptions in Corollary \ref{cor:hypernlpbound-mersp} hold if $d$ and $\gamma$ are determined by any of the three strategies and $p$ is determined by Theorem \ref{cor:best-p-nlp}; that is, hNLP-Id, hNLP-Di, and hNLP-Tr are convex relaxations of \ref{hNLP}.

 Coming back to \ref{MERSP}, suppose that the covariance matrix $C$ is positive definite. Then both $B$ and $B_T$ are positive definite, and moreover $B_T^{-1} \succeq B^{-1}$. Hence, by setting $C_1:= B^{-1}$, $C_2:= B_T^{-1}$, and choosing $\psi\leq 1$,  we obtain $C_2\succeq \psi C_1$\,. Therefore,  for any given $\Phi\in\mathbb{R}^n_{++}$\,, if the parameters $\gamma$, $d$, and $p$  satisfy the remaining assumptions of Corollary \ref{cor:hypernlpbound-mersp},   
then \ref{hNLP} is a convex optimization problem. In what follows, we refer to this formulation as \NLPcompMERSP,
 which we 
call the ``complementary hyper-scaled NLP bound'' (complementary hNLP bound) for \ref{MERSP}. Specifically, this last bound is given by ${\znlpthing}(B^{-1}\,,B_T^{-1}\,,n-s\,;\psi,\Phi) +\ldet(B) -\ldet(B_T)$.
Furthermore, in the particular case where  $\psi:=1$ and $\Phi:=\mathbf{e}$, complementary hNLP bound reduces exactly to the complementary NLP bound for \ref{MERSP} from \citep*{AFLW_Remote}. 

Our main motivation in what follows, after the next section, is to further exploit the \ref{hNLP} bound for \ref{MEDP} by allowing $\psi \neq 1$ and $\Phi \neq \mathbf{e}$, with the aim of strengthening the bound while preserving the convexity of the resulting relaxation. Moreover, specifically for \ref{MERSP}, we will demonstrate that this framework remains applicable when $B$ and $B_T$ are singular.  Before pursuing these developments, in the next section we compare the NLP bound for \ref{MERSP} with another bound also available in the literature.  


\section{Comparison between MERSP bounds}\label{sec:compare-mersp-bounds}

In this section, we theoretically compare the complementary NLP bound for \ref{MERSP}, 
with another known upper bound, namely
the spectral bound for \ref{MERSP}.
We demonstrate dominance of the  complementary NLP bound for \ref{MERSP}, under some
weak conditions. 
In doing so, we 
can see that our \ref{hNLP} bound dominates the spectral bound  under the same conditions.

\citep*{AFLW_Remote} introduced the \emph{spectral bound} for \ref{MERSP} when $C \succ 0$, given by
    \begin{equation}\label{sb}
    \hypertarget{zspectraltarget}{\zspectralthing}(B,B_T\,,s) := \min_{\eta \in \mathbb{R}^n}  v(\eta):=
    \textstyle\sum_{\ell=1}^s\log \lambda_\ell( D_{\eta} B D_{\eta})- \log \lambda_{n-\ell+1}( D_{\eta} B_T D_{\eta}),
    \end{equation}
    where  
    $D_{\eta}\in \mathbb{S}^n_{++}$ is  the diagonal matrix defined by
    $
    D_{\eta}[i,i]:= 
        \exp\left\{ \eta_i/2\right\}$,  for $i \in N$.

 Recently, \citep*{MESP2DOPT} established that the NLP bound for \ref{MESP} with the $\IdentityNLP$\ strategy dominates the spectral bound for \ref{MESP} when $s$ is less than the multiplicity of the largest eigenvalue of $C$. Motivated by this result, we  demonstrate  that, under suitable conditions,
 the  complementary  NLP bound for \ref{MERSP} dominates the spectral bound for \ref{MERSP}. 
\begin{theorem}\label{thm:nlp_dom_spec}
    If $C \succ 0$ and any of the following successively weaker conditions holds,
    \begin{enumerate}
    \item $t \leq s$,
    \item $\rank(C[N,T]) \leq s$,
    \item   $\textstyle\sum_{\ell=1}^{n-s}\left(\log \lambda_{n-\ell+1}( D_{
        \hat\eta} B D_{\hat\eta})  - \log \lambda_{\ell}( D_{\hat\eta} B_T D_{\hat\eta} )\right) \leq 0$, where $\hat\eta := \argmin v(\eta)$ (see \eqref{sb}),
    \end{enumerate}
    \smallskip
 then, for any of the parameters choices considered -- namely {\rm\IdentityNLP}, {\rm\DiagonalNLP}, or {\rm\TraceNLP}  -- the  complementary  {\rm NLP} bound dominates the spectral bound for {\rm\ref{MERSP}}. Specifically, we have 
     \[
     \znlp(B^{-1},B_T^{-1},n-s;1,\mathbf{e}) + \ldet(B) - \ldet(B_T) \leq  \zspectral(B,B_T\,,s).
     \]
\end{theorem}

\begin{proof} 
    Let $C_1 := B^{-1}$ and $C_2 := B_T^{-1}$. Let $M_k(x) := \Diag((\gamma d)^x) +  
\gamma\Diag(x^{p/2})(C_k-D)\Diag(x^{p/2}))$ for $k=1,2$. The  complementary NLP bound for \ref{MERSP}, is given by
\[
 \znlp(B^{-1},B_T^{-1}\,,n-s\,;1,\mathbf{e}) + \ldet (B) - \ldet(B_T) =  \ldet(M_1(x)) - \ldet(M_2(x)) + \ldet(B) - \ldet(B_T).
\]
 Notice that $C_2 \succeq C_1$\,. Then, as all other  conditions in Theorem \ref{thm:nlpbound-mersp} hold for the parameters choices considered,  we have that $M_2(x) \succeq M_1(x)$, for all $x \in [0,1]^n$. Thus, from \citep*[Corollary 7.7.4 (b)]{HJBook}, we have that $\ldet(M_1(x)) - \ldet(M_2(x)) \leq 0$.

Next, for the spectral bound, we have that
\begin{align*}
        \zspectral&(B,B_T\,,s) = \textstyle\sum_{\ell=1}^s\log \lambda_\ell\left( D_{\hat\eta}BD_{\hat\eta}\right)- \log \lambda_{n-\ell+1}\left( D_{\hat\eta}B_TD_{\hat\eta}  \right)\\
        &\!\!\!\!\!\!=\ldet\left( D_{\hat\eta}BD_{\hat\eta}\right)- \ldet\left(D_{\hat\eta} B_T D_{\hat\eta}  \right)+ \textstyle\sum_{\ell=s+1}^n \log \lambda_{n-\ell+1}\left( D_{\hat\eta}B_TD_{\hat\eta}  \right)-\log \lambda_\ell\left( D_{\hat\eta}BD_{\hat\eta}\right)\\
        &\!\!\!\!\!\!=\ldet\left( B\right)- \ldet\left( B_T  \right) + \textstyle\sum_{\ell=1}^{n-s}\log \lambda_{\ell}\left( D_{\hat\eta}B_TD_{\hat\eta}  \right)-\log \lambda_{n-\ell+1}\left( D_{\hat\eta}BD_{\hat\eta}\right).
    \end{align*}
    So, the difference between the complementary NLP bound and the spectral bound for \ref{MERSP} is given by
    \begin{align*}
    &\ldet(M_1(x)) - \ldet(M_2(x)) + \textstyle\sum_{\ell=1}^{n-s}\log \lambda_{n-\ell+1}\left( D_{\hat\eta}BD_{\hat\eta}\right)  - \log \lambda_{\ell}\left( D_{\hat\eta}B_TD_{\hat\eta}  \right)\\
        &\qquad\leq \textstyle\sum_{\ell=1}^{n-s}\log \lambda_{n-\ell+1}\left( D_{\hat\eta}BD_{\hat\eta}\right)  - \log \lambda_{\ell}\left( D_{\hat\eta}B_TD_{\hat\eta}  \right).
    \end{align*}
    Finally, considering that we satisfy the condition in:
    \begin{itemize}
        \item item 3, the result trivially follows;
        \item item 2, then
        \begin{align*}
        s \geq \rank(C[N,T]) &= \rank(C[N,T](C[T,T])^{-1/2}) = \rank(C[N,T](C[T,T])^{-1}C[T,N])\\ 
        &= \rank(C[N,N] - C_T[N,N]) =   \rank(D_{\hat\eta}(B-  B_T)D_{\hat\eta}).
        \end{align*}
        Let $Q \Theta Q^\top$ be the eigendecomposition of $D_{\hat\eta}(B- B_T)D_{\hat\eta}$ with $Q := \begin{bmatrix}\underset{\scriptscriptstyle n\times s}{Q_1}& \!\underset{\scriptscriptstyle n\times (n-s)}{Q_2}\end{bmatrix}$ and $\Theta:=\Diag(\theta_1\,,\theta_2\,,\dots,\theta_n)$ with $\theta_1 \geq \theta_2 \geq \dots \geq \theta_{s} \geq \theta_{s+1} = \theta_{s+2} = \dots =  \theta_n = 0$\,. Then  $Q_2^\top D_{\hat\eta}BD_{\hat\eta} Q_2 =  Q_2^\top D_{\hat\eta} B_TD_{\hat\eta} Q_2$\,.
        From \citep*[Corollary 4.3.16]{HJBook}, for $\ell = 1,\dots,n-s$,  we have that
        \begin{align*}
            &\lambda_\ell(D_{\hat\eta}B_TD_{\hat\eta}) \geq \lambda_{\ell}(Q_2^\top D_{\hat\eta}B_TD_{\hat\eta}Q_2) = \lambda_{\ell}(Q_2^\top D_{\hat\eta}BD_{\hat\eta}Q_2)\geq \lambda_{s+\ell}(D_{\hat\eta}BD_{\hat\eta})\Rightarrow\\
            &\log\lambda_\ell(D_{\hat\eta}B_TD_{\hat\eta}) \geq \log\lambda_{s+\ell}(D_{\hat\eta}BD_{\hat\eta}),
        \end{align*}
         then 
         $
         \sum_{\ell=1}^{n-s}\log \lambda_{n-\ell+1}\left( D_{\hat\eta}BD_{\hat\eta}\right)  - \log \lambda_{\ell}\left( D_{\hat\eta}B_TD_{\hat\eta}  \right) \leq 0$
         and the result follows from item 3;
         \item item 1, then, as $t \leq s \leq n$, we have that  $\rank(C[N,T])\leq s$ and the result follows from item 2. \qed
    \end{itemize}
\end{proof}


\section{\texorpdfstring{Computing the parameter $\psi$}{Computing the parameter psi}}\label{sec:aug-nlp-bound}

In this section, we consider 
 $\Phi$ fixed in \ref{hNLP}. Without loss of generality, we set $\Phi:=\mathbf{e}$, and we investigate how to select $\psi$, with the goal of obtaining the best (least) possible bound, while assuring the convexity of \ref{hNLP}. 

As mentioned earlier,
\hypertarget{NLPcompMERSP}{\NLPcompMERSP}
refers to the hyper-scaled \ref{hNLP} relaxation using 
$C_1:=B^{-1}$ and $C_2:=B_T^{-1}$, which we 
call the \emph{complementary hyper-scaled NLP bound} for \ref{MERSP}.
Likewise, 
\hypertarget{NLPorigMERSP}{ \NLPorigMERSP}\, 
denotes the \ref{hNLP} relaxation using 
$C_1:=B$ and $C_2:=B_T$\,,
 which we 
call the \emph{(original) hyper-scaled NLP bound} for \ref{MERSP}.
Note that when $C \succ 0$,
we have $B_T^{-1}\succeq \psi B^{-1}$ for all $\psi\leq 1$. Furthermore,  as $B$ and $B_T$ are positive definite in this case,  there always exists some $\psi\! > \!0$ such that $B_T \succeq \psi B$. In the next section, we will demonstrate that a similar relation can also be established when  $C\succeq 0$ under additional assumptions, namely, there exists a $\psi > 0$ such that $B_T \succeq \psi B$.

Finally, given $C_1$\,, $C_2$ and $\psi>0$ satisfying $C_2\succeq \psi C_1$\,,  it is possible to choose the parameters $\gamma$, $d$, and $p$ in \ref{hNLP} so that all remaining assumptions in Corollary \ref{cor:hypernlpbound-mersp} are satisfied, thereby ensuring the convexity of  \ref{hNLP}; see, for instance, the 
$\IdentityNLP$, $\DiagonalNLP$ and $\TraceNLP$ parameter-selection strategies. Hence, in the following, we will regard the condition $C_2\succeq \psi C_1$ as sufficient to guarantee the convexity of either \NLPcompMERSP\, or \NLPorigMERSP. 


\subsection{Exact convex relaxation for singular covariance matrices}\label{subsec:singular-case-thms}
In \citep*{AFLW_Remote}, the analysis was restricted to \ref{MERSP} instances with a positive-definite covariance matrix $C$, which directly implies that $B:=C[N,N]$ and $B_T:=C_T[N,N]$ 
are positive definite. As $B_T^{-1}\succeq B^{-1}$, this allowed  the replacement of $C_1$ and $C_2$ respectively by $B^{-1}$ and $B_T^{-1}$ in \ref{hNLP} with $\psi:=1$ and $\Phi:=\mathbf{e}$ leading to a convex relaxation and associated  complementary  NLP bound for \ref{MERSP}. 

Here, we consider a more general class of \ref{MERSP} instances in which $C$ may be singular and we investigate a convex relaxation for this class of instances  also considering the more-general problem \ref{MEDP} and its relaxation \ref{hNLP}. 
First, we observe that \ref{MERSP} is well defined if, for any set $S$ with 
$|S|=s$ such that $C[S,S] \succ 0$, we have that $C_T[S,S]\succ 0$ ---  otherwise \ref{MERSP}  would be unbounded. 
We will present a sufficient condition under which \ref{MERSP} is well defined for all $0<s<n$ and demonstrate that the  hyper-scaled  NLP bound for \ref{MEDP} provides an exact convex relaxation for \ref{MERSP} in this case, even when the covariance matrix $C$ is singular and the complementing principle cannot be applied.

\begin{lemma}\label{lem:exact-rel-singular-case}
    Let $C\succeq 0$. There exists $\psi>0$ such that
$C_T[N,N]\succeq\psi\,C[N,N]$ if and only if
\begin{equation}\label{eq:cond-psi-geq0-exists}
    C[T,T]-C[N,T]^{\top}C[N,N]^{\dagger}C[N,T]\succ 0.
\end{equation}
Moreover, when \eqref{eq:cond-psi-geq0-exists} holds, the relation
$C_T[N,N]\succeq\psi\,C[N,N]$ is satisfied for every
\begin{equation}\label{eq:psi_range}
    \psi\in\left(0,1-\sigma_{1}^2\left(C[T,T]^{-1/2}C[T,N]C[N,N]^{\dagger/2}\right)\right],
\end{equation}
and this interval is nonempty.
\end{lemma}

\begin{proof}
Let $E:=C[T,T]^{-1/2}C[T,N]C[N,N]^{\dagger/2}$. Then,
\begin{equation}\label{eq:svd_eig_relation_singular_case}
    \sigma_{1}^2(E) = \lambda_1(EE^\top) =  \lambda_1\left(C[T,T]^{-1/2}C[N,T]^\top (C[N,N])^\dagger C[N,T] C[T,T]^{-1/2}\right),
\end{equation}
and
\begin{equation}\label{eq:lmax}
    \begin{array}{ll}
    &C[T,T] - C[N,T]^{\top}(C[N,N])^{\dagger}C[N,T]\succ 0 \Leftrightarrow\\
         &I_t - C[T,T]^{-1/2}C[N,T]^{\top}(C[N,N])^{\dagger}C[N,T]C[T,T]^{-1/2}\succ 0 \Leftrightarrow  \\
         &\sigma_{1}^2(E) < 1.
    \end{array}
\end{equation}
For $\psi\in(0,1)$, we have
\begin{equation}\label{eq:psi_equivalence_singular_case_iff}
    \begin{array}{ll}
         &C_T[N,N] \succeq \psi C[N,N] \Leftrightarrow\\[3pt]
       &(1-\psi)C[N,N] - C[N,T](C[T,T])^{-1}C[T,N]\succeq 0 \Leftrightarrow\\[3pt]
       &C[N,N] - C[N,T]((1-\psi)C[T,T])^{-1}C[T,N]\succeq 0 \Leftrightarrow\\[3pt]
       &\begin{pmatrix}
           C[N,N] & C[N,T]\\
           C[N,T]^{\top} & (1-\psi)C[T,T]
       \end{pmatrix}\succeq 0 \Leftrightarrow\\[9pt]
       &(1-\psi)C[T,T] - C[N,T]^{\top}(C[N,N])^{\dagger}C[N,T]\succeq 0\Leftrightarrow\\[3pt]
        &(1-\psi) I_t - C[T,T]^{-1/2}C[N,T]^\top (C[N,N])^\dagger C[N,T] C[T,T]^{-1/2}\succeq 0 \Leftrightarrow\\[3pt]
        &\sigma_{1}^2(E) \leq 1-\psi,
    \end{array}
\end{equation}
where the fourth equivalence follows from \citep*[Theorem 1.20]{SchurBook} and the last from \eqref{eq:svd_eig_relation_singular_case}. 
\begin{itemize}
   \item[$(i)$] Assume that there exists a $\hat{\psi}>0$ such that $C_T[N,N]\succeq\hat\psi C[N,N]$. Then, because $C[N,N]\succeq 0$, we have that $C_T[N,N]\succeq\psi C[N,N]$ for every $\psi\le\hat\psi$. Therefore, there exists  $\tilde \psi\in(0,1)$ such that  $C_T[N,N]\succeq\tilde \psi C[N,N]$.
  By \eqref{eq:psi_equivalence_singular_case_iff}, $\sigma_1^2(E)\leq 1-\tilde\psi$, thus $\sigma_1^2(E)<1$. Hence, by \eqref{eq:lmax}, \eqref{eq:cond-psi-geq0-exists} holds.

    \item[$(ii)$] Assume that  \eqref{eq:cond-psi-geq0-exists} holds. Then, by \eqref{eq:lmax},  $\sigma_1^2(E)<1$. 
   Let $\psi\in(0,1-\sigma_1^2(E)]$, which is a nonempty interval.  If $\psi=1$, then 
   $\sigma_1^2(E)=0$, and because $C\succeq 0$, we have that $C[N,T]=0$, and hence,  $C_T[N,N]=C[N,N]$. 
   If $\psi<1$,  by  \eqref{eq:psi_equivalence_singular_case_iff},  $C_T[N,N]\succeq\psi C[N,N]$. Therefore, $C_T[N,N]\succeq\psi C[N,N]$ for all $\psi\in(0,1-\sigma_1^2(E)]$.
\end{itemize}
\qed\end{proof}

\begin{theorem}\label{thm:good-case-singular}
    Let $C \succeq 0$. If \eqref{eq:cond-psi-geq0-exists} holds, then for any set $S \subset N$ with $0 < |S| < n$, we have 
        $
        C[S,S] \succ 0~\Leftrightarrow~C_T[S,S]\succ 0
        $.
\end{theorem}

\begin{proof} 
Assume that $C_T[S,S] \succ 0$. Let $W := \left(\begin{smallmatrix}
        C[S,S] & C[S,T]\\
        C[T,S] & C[T,T]
    \end{smallmatrix}\right)$ and note that $C_T[S,S]$ is the Schur complement of
    $C[T,T]$ in $W$. 
    As $C[T,T] \succ 0$ and $C_T[S,S] \succ 0$, then by \citep*[Theorem 1.12]{SchurBook}, we have that $W \succ 0$. Then, as $C[S,S]$ is a principal submatrix of $W$, we conclude that $C[S,S] \succ 0$.
    
    Now, 
    from Lemma \ref{lem:exact-rel-singular-case}, we have that there exists $\psi > 0$ such that 
    $C_T[N,N] \succeq \psi C[N,N]$, which implies that $C_T[S,S] \succeq \psi C[S,S]$ and because $\psi > 0$ and $C[S,S] \succ 0$, we conclude that $C_T[S,S] \succ 0$.  
   \qed\end{proof}
   
Next, we demonstrate that the sufficient condition in Theorem \ref{thm:good-case-singular} is not necessary. We demonstrate
that it is possible to have $C[S,S] \succ 0$ and $C_T[S,S]\succ 0$ for any set $S$ with 
$|S|=s$, without condition \eqref{eq:cond-psi-geq0-exists} being satisfied.
\begin{example}
    For $C = \left(\begin{smallmatrix}
        1 & 0 & 1\\
        0 & 1 & 1\\
        1 & 1 & 2
    \end{smallmatrix}\right)$, $N :=\{1,2\}$, $T:=\{3\}$, condition   \eqref{eq:cond-psi-geq0-exists} does not hold, and for any set $S \subset N$ with $|S|=1$, we have  $C[S,S] \succ 0$ and $C_T[S,S]\succ 0.$
     \hfill $\clubsuit$
\end{example}

We note, however, that there is a particular situation where condition  \eqref{eq:cond-psi-geq0-exists} is necessary, as we observe in the following proposition. 

\begin{proposition}\label{rem:necessary-singular}
    If \eqref{eq:cond-psi-geq0-exists} is not satisfied and $\rank(C[N,N]) < n$, then there always exists a set $S \subset N$ with $|S|=\rank(C[N,N])$ such that $C[S,S] \succ 0$ and $C_T[S,S]\not\succ 0$.
\end{proposition}

\begin{proof} 
  As $C \succeq 0$ and condition \eqref{eq:cond-psi-geq0-exists} is not satisfied, we have  
  \[
  C[T,T] - C[N,T]^{\top}(C[N,N])^{\dagger}C[N,T]\succeq 0
  \]
  and singular. From the well-known result
        \[
        \rank(C) = \rank(C_T[N,N]) + \rank(C[T,T]),
        \]
        (see for example, \citep*[Eq. 0.9.2]{SchurBook}), and from  
         \citep*[Page 43]{SchurBook}, we have that 
        \begin{align*}
            \rank(C) =\rank(C[N,N]) + \rank(C[T,T] - C[N,T]^{\top}(C[N,N])^{\dagger}C[N,T]).
        \end{align*}
        As $\rank(C[T,T]) = t$  and $\rank(C[T,T] - C[N,T]^{\top}(C[N,N])^{\dagger}C[N,T] ) \leq t-1$, we have that $\rank(C_T[N,N]) \leq \rank(C[N,N]) -1$. 
        
        By our assumption, we have that $\rank(C[N,N]) < n$. Choose a set $S$, with $|S| = \rank(C[N,N])$, so that $C[S,S] \succ 0$.       
    Then we have $\rank(C_T[S,S]) \leq \rank(C_T[N,N]) \leq \rank(C[N,N]) - 1 < |S|$.
\qed\end{proof}

\begin{remark}
    For any {\rm\ref{MERSP}} instance satisfying condition \eqref{eq:cond-psi-geq0-exists}, there exists a $\psi > 0$ such that $B_T \succeq \psi B$.  Therefore, under this sufficient condition,  {\rm\NLPorigMERSP}  
 with an appropriate selection of the remaining parameters, provides a convex relaxation for {\rm\ref{MERSP}}, even when the covariance matrix $C$ is singular.
\end{remark}


\subsection{\texorpdfstring{Optimal choice of the parameter $\psi$}{Optimal choice of the parameter psi}}

In this subsection, we demonstrate that finding the parameter $\psi$ that minimizes the bound admits a closed-form solution. For this, we provide the following lemma, which generalizes \citep*[Lemma 22]{MESP2DOPT}. We note that the  lemma concerns only the first term of the objective function of \ref{hNLP} and is therefore independent of $C_2$\,. 

\begin{lemma}\label{lem:nlp_aug_psi_decrease}
    Let $C_1 \in \mathbb{S}^{n}_{+}$\,,   $ 0<s<n$, and $\hat{x} \in [0,1]^n$ with  $\mathbf{e}^\top \hat{x} = s$. Let $D:=\Diag(d)$ for some $d\in \mathbb{R}^n_{++}$\,,  and let $\gamma>0$ and  $p \geq \mathbf{e}$ be such that 
     $0 < \gamma d_i \leq \exp(p_i-\sqrt{p_i})$,  for all $i \in N$. For each $\psi \in (0,1/\lambda_{1}(D^{-1/2}C_1D^{-1/2})]$, define
    \[
L(\psi):=\Diag((\gamma d)^{\hat{x}}) +  
\gamma\Diag(\hat{x}^{p/2})(\psi C_1-D)\Diag({\hat{x}}^{p/2}),
\]
 and assume that $L(\psi)\succ 0$ for all such $\psi$.  Then, $f_1(\hat{x};\psi,\mathbf{e})$ is {nonincreasing} in $\psi$ on $(0,1/\lambda_{1}(D^{-1/2}C_1D^{-1/2})]$.
Moreover, if $\hat{x} \notin \{0,1\}^n$, then $f_1(\hat x;\psi,\mathbf{e})$ is {strictly decreasing} in $\psi$ on $(0,1/\lambda_{1}(D^{-1/2}C_1D^{-1/2}))$.
\end{lemma}

\begin{proof} 
    Let $G := \gamma\Diag(\hat{x}^{p/2}) C_1\Diag(\hat{x}^{p/2})$. 
For $\Phi=\mathbf{e}$, we have that  $
    \frac{\partial f_1}{\partial{\psi}}= \Tr(L(\psi)^{-1}G) - s/\psi$. 
    Then, it suffices to demonstrate that $\Tr(L(\psi)^{-1}\psi G) \leq s$, with strict inequality whenever  $\hat{x}\notin \{0,1\}^n$. 
    Note that 
    \[
    \psi G= L(\psi) -  \Diag((\gamma d)^{\hat{x}}) +  
\gamma\Diag(\hat{x}^{p/2})D\Diag({\hat{x}}^{p/2}).
    \]
    Define $\alpha_j := (L(\psi)^{-1})_{jj} -(\gamma d_j)^{-{\hat{x}}_j}$ and $\beta_j := (L(\psi)-\psi G)_{jj}=(\gamma d_j)^{\hat{x}_j} - \gamma d_j \hat{x}_j^{p_j}$ for $j \in N$. Then, 
\begin{align*}
    \Tr\left(L(\psi)^{-1}(L(\psi) - \psi G)\right) &=\textstyle\sum_{j\in N}(L(\psi)^{-1})_{jj}\beta_j
    = \textstyle\sum_{j\in N}((\gamma d_j)^{-{\hat{x}}_j} + \alpha_j)\beta_j\\ 
    &= \Tr\left(\Diag((\gamma d)^{\hat{x}})^{-1}(L(\psi) - \psi G)\right) + \textstyle\sum_{j \in N} \alpha_j\beta_j\,.
\end{align*}
Then,
    \begin{align*}
         \Tr&(L(\psi)^{-1} \psi G)  =  n- \Tr\left(L(\psi)^{-1}(L(\psi) - \psi G)\right)\\
         &=n-\Tr\left(\Diag((\gamma d)^{\hat{x}})^{-1}(L(\psi) - \psi G)\right) - \textstyle\sum_{j \in N} \alpha_j\beta_j\\
         &= n - 
         \Tr\big(\Diag((\gamma d)^{\hat{x}})^{-1}(\Diag((\gamma d)^{\hat{x}}) - \gamma \Diag(\hat{x}^{p/2}) D\Diag(\hat{x}^{p/2}))\big) - \textstyle\sum_{j\in N} \alpha_j \beta_j\\
         &= n - 
         \Tr\big(I_n - \Diag((\gamma d)^{\hat{x}})^{-1}\gamma \Diag(\hat{x}^{p/2}) D\Diag(\hat{x}^{p/2})\big) - \textstyle\sum_{j\in N} \alpha_j \beta_j\\
         &= \Tr\big(\Diag((\gamma d)^{\hat{x}})^{-1}\gamma \Diag(\hat{x}^{p/2}) D\Diag(\hat{x}^{p/2})\big)- \textstyle\sum_{j\in N} \alpha_j \beta_j\\
         &=\textstyle\sum_{i \in N} (\gamma d_i)^{1-\hat{x}_i}\hat{x}_i^{p_i}- \textstyle\sum_{j\in N} \alpha_j \beta_j\,.
    \end{align*}
    Next, we demonstrate that 
    \begin{itemize}
        \item $\textstyle\sum_{i \in N} (\gamma d_i)^{1-\hat{x}_i}\hat{x}_i^{p_i} \leq s$.
        
For $i \in N$, if $\hat{x}_i \in \{0,1\}$, then $(\gamma d_i)^{1-\hat{x}_i}\hat{x}_i^{p_i} = \hat{x}_i$\,. If $\hat{x}_i \in (0,1)$, we will demonstrate that 
    $(\gamma d_i)^{1-\hat{x}_i}\hat{x}_i^{p_i} \leq \hat{x}_i$ . Equivalently, we will demonstrate that $(1-\hat{x}_i)\log(\gamma d_i) + (p_i-1)\log(\hat{x}_i) \leq 0$. As $0 < \gamma d_i \leq \exp(p_i-\sqrt{p_i})$, then $\log(\gamma d_i) \leq p_i - \sqrt{p_i}$\,, so it suffices to demonstrate that for
    \[
    h(y) := (1-y)(p_i - \sqrt{p_i}) + (p_i-1)\log(y) 
    \]
    we have $h(y) < 0$ for $y \in (0,1)$. We note that $h(1) = 0$ and that
    \[
    h'(y) = - (p_i-\sqrt{p_i}) + (p_i-1)/y.
    \]
    As $(p_i-1)/y \geq p_i - 1 \geq p_i -\sqrt{p_i}$\,, $h$ is nondecreasing for $y \in (0,1)$ so we conclude that $h(y) \leq 0$ for $y \in (0,1)$. Then, 
    \[(1-\hat{x}_i)\log(\gamma d_i) + (p_i-1)\log(\hat{x}_i) \leq h(\hat{x}_i) \leq 0.\]
    Therefore, we have
    \[
    \textstyle\sum_{i \in N} (\gamma d_i)^{1-\hat{x}_i}\hat{x}_i^{p_i} \leq \textstyle\sum_{i \in N} \hat{x}_i = s.
    \]
    \item $\textstyle\sum_{j\in N} \alpha_j \beta_j \geq 0$  if $\hat{x} \in [0,1]^n$.
    
     Note that for $0 < \psi \leq 1/\lambda_1(D^{-1/2} C_1 D^{-1/2})$, we have $\psi C_1 \preceq D \Rightarrow L(\psi) \preceq \Diag((\gamma d)^{\hat{x}}) \Rightarrow L(\psi)^{-1} \succeq \Diag((\gamma d)^{\hat{x}})^{-1} \Rightarrow \alpha \geq 0$.

    For $j \in N$, we have
    $\beta_j := (\gamma d_j)^{\hat{x}_j} - \gamma d_j \hat{x}_j^{p_j} = (\gamma d_j)^{\hat{x}_j}(1-(\gamma d_j)^{1-\hat{x}_j}\hat{x}_j^{p_j}) \geq 0$, as we proved   that $(\gamma d_j)^{1-\hat{x}_j}\hat{x}_j^{p_j} \leq \hat{x}_j \leq 1$.\\

    \item $\alpha_{\hat\jmath} \beta_{\hat\jmath} > 0$  if $\hat{x}_{\hat\jmath} \in (0,1)$ where ${\hat\jmath} \in N$.

    Note that for $0 < \psi < 1/\lambda_1(D^{-1/2} C_1 D^{-1/2})$, we have  $\psi C_1 \prec D \Rightarrow \psi (C_1)_{\hat\jmath\hat\jmath} < d_{\hat\jmath}$\,. Then, 
    \[
    L(\psi)_{\hat\jmath\hat\jmath} =(\gamma d_{\hat\jmath})^{\hat{x}_{\hat\jmath}} +  
\gamma \hat{x}^{p_{\hat\jmath}}_{\hat\jmath}(\psi (C_1)_{\hat\jmath\hat\jmath} - d_{\hat\jmath}) \Rightarrow L(\psi)_{\hat\jmath\hat\jmath} < (\gamma d_{\hat\jmath})^{\hat{x}_{\hat\jmath}} \Rightarrow 1/L(\psi)_{\hat\jmath\hat\jmath} \,>\, (\gamma d_{\hat\jmath})^{-\hat{x}_{\hat\jmath}}\,.
    \]
    
    We note that $L(\psi)\circ L(\psi)^{-1}\succeq I_n$ (see, for example, \citep*[Theorem 7.7.9(c)]{HJBook}), which implies that $(L(\psi)^{-1})_{\hat\jmath\hat\jmath} \geq 1/L(\psi)_{\hat\jmath\hat\jmath} > (\gamma d_{\hat\jmath})^{-\hat{x}_{\hat\jmath}} \Rightarrow \alpha_{\hat\jmath} > 0$.

    We also note that
    $\beta_{\hat\jmath} := (\gamma d_{\hat\jmath})^{\hat{x}_{\hat\jmath}} - \gamma d_{\hat\jmath} \hat{x}_{\hat\jmath}^{p_{\hat\jmath}} = (\gamma d_{\hat\jmath})^{\hat{x}_{\hat\jmath}}(1-(\gamma d_{\hat\jmath})^{1-\hat{x}_{\hat\jmath}}\hat{x}_{\hat\jmath}^{p_{\hat\jmath}}) > 0$, as we proved   that $(\gamma d_{\hat\jmath})^{1-\hat{x}_{\hat\jmath}}\hat{x}_{\hat\jmath}^{p_{\hat\jmath}} \leq \hat{x}_{\hat\jmath} < 1$, and because $\gamma d > \mathbf{0}$. 
    \end{itemize}
    The result follows.
\qed\end{proof}

\begin{theorem}\label{thm:best_alpha_NLPId} 
     Assume that $C\succ 0$. Then the least objective value for  {\rm\NLPcompMERSP} is obtained with $\psi:=\psi^*$, where \begin{equation}\label{psistar_Comp}
\psi^*:=1/\lambda_{1}\left(C_T[N,N]^{1/2}C[N,N]^{-1}C_T[N,N]^{1/2}\right).
     \end{equation}
     Assume that $C\succeq 0$ and that condition \eqref{eq:cond-psi-geq0-exists} is satisfied. Then the least objective value for {\rm hNLP}$\left(\mbox{\rm MERSP}\right)$ is obtained with $\psi:=\psi^*$, where \begin{equation}\label{psistar_orig}
     \psi^*:=  1-\sigma_{1}^2\left(C[T,T]^{-1/2}C[T,N]C[N,N]^{\dagger/2}\right).
     \end{equation}
\end{theorem}

\begin{proof} 
From Lemma \ref{lem:nlp_aug_psi_decrease}, we see that to obtain the
least objective value
for \ref{hNLP}, it suffices to choose the largest $\psi$ such that $D \succeq C_2 \succeq \psi C_1$\,.  If $C_2 \succ 0$,  the largest admissible value of $\psi$ is given by  $\psi^*:= 1/\lambda_{1}(C_2^{-1/2}C_1C_2^{-1/2})$.
Considering the choices $C_1:=B^{-1}:= C[N,N]^{-1}$, and $C_2:=B_T^{-1}:=C_T[N,N]^{-1}$ associated with  \NLPcompMERSP, we obtain \eqref{psistar_Comp}.  
Considering instead
{\rm hNLP}$\left(\mbox{\rm MERSP}\right)$, associated with the choices $C_1:=B:=C[N,N]$ and $C_2:=B_T:=C_T[N,N]$, in case $C$ is singular, $C_2$ may be singular. Then the greatest admissible value of $\psi$ in \eqref{psistar_orig}  is given by Lemma \ref{lem:exact-rel-singular-case}. 
\qed\end{proof}

When $C \succ 0$, both $C[N,N]$ and $C_T[N,N]$ are positive definite. In this case, \eqref{psistar_orig} reduces to an analogous expression to \eqref{psistar_Comp}. This observation is formalized in the following corollary.

\begin{corollary}\label{cor:same-scaling-pos-def-case}
    For  $C \succ 0$,  we have that 
    \[
     1-\sigma_{1}^2\left(C[T,T]^{-1/2}C[T,N]C[N,N]^{\dagger/2}\right) = 1/\lambda_{1}\left(C_T[N,N]^{-1/2}C[N,N]C_T[N,N]^{-1/2}\right).
    \] 
\end{corollary}

\begin{proof} 
    Let $C_1 := C[N,N]$, $C_2:= C_T[N,N]$ and  
    \[
    E:=C[T,T]^{-1/2}C[T,N]C[N,N]^{-1/2},
    \]
    and note that 
    \[
    C_2 = C_1 - C[T,N]^\top C[T,T]^{-1}C[T,N]=    C_1^{1/2}( I - E^\top E) C_1^{1/2}\,.
    \]
  As both $C_1$ and $C_2$ are positive definite matrices with rank $n$, we have $\rank(I_n - E^\top E) = n$, and so we have that $C_2^{-1} = C_1^{-1/2}(I_n - E^\top E)^{-1}C_1^{-1/2}$\,. Then, 
    \begin{align*}
        1/\lambda_{1}(&C_2^{-1/2}C_1C_2^{-1/2}) = 1/\lambda_{1}(C_1^{1/2}C_2^{-1}C_1^{1/2}) = 1/\lambda_{1}((I_n - E^\top E)^{-1}) = \lambda_n(I_n - E^\top E)\\
        &= 1 - \lambda_1(E^\top E) = 1 - \sigma_1^2(E) = 1-\sigma_{1}^2(C[T,T]^{-1/2}C[T,N]C[N,N]^{-1/2}).
    \end{align*}
    The result follows.
\qed\end{proof}

\begin{corollary}\label{cor:CNT_full_row_rank}
    Let $C \succ 0$, and let $\psi:=\psi^*$ be defined as in \eqref{psistar_Comp}. Then,  
    \begin{itemize}
        \item we do not get a bound improvement 
         from
{\rm \NLPcompMERSP} (compared to the 
 complementary  NLP bound from \citep*{AFLW_Remote}),
using $\psi\not=1$, 
        if  $\rank(C[N,T])<n$; 
        \item we get a bound improvement 
         from
 {\rm \NLPcompMERSP} (compared to the 
 complementary  NLP bound from \citep*{AFLW_Remote}),
using $\psi\not=1$, if  $\rank(C[N,T])=n$ and there exists a non-binary optimal solution to {\rm\NLPcompMERSP}.
    \end{itemize}
\end{corollary}

\begin{proof} 
    Note that  
    \[
    \rank(C[N,T](C[T,T])^{-1}C[T,N]) = \rank(C[N,T](C[T,T])^{-1/2}) = \rank(C[N,T]),
    \]
    and 
    \[
    C[N,N] - C_T[N,N] = C[N,T](C[T,T])^{-1}C[T,N]\succeq 0.
    \]
    Then, we have that
    \begin{align*}
    &\rank(C[N,T]) = n \Leftrightarrow C[N,N] \succ  C_T[N,N] \Leftrightarrow C_T[N,N]^{-1}\succ C[N,N]^{-1} \Leftrightarrow \\&I_n \succ C_T[N,N]^{1/2}C[N,N]^{-1}C_T[N,N]^{1/2} \Leftrightarrow \\
    &\lambda_{1}(C_T[N,N]^{1/2}C[N,N]^{-1}C_T[N,N]^{1/2}) < 1 \Leftrightarrow \psi^*>1.
    \end{align*}
    Then, if $\rank(C[N,T])<n$, we have that $\psi^*=1$, and the result for this case follows.
    Now, consider the case  of $\rank(C[N,T])=n$. 
    
     Let $\tilde{x}$ be an optimal solution for \NLPcompMERSP\, when $\psi=1$. Let $\hat{x}$ be a non-binary optimal solution for \NLPcompMERSP\, when $\psi:=\psi^*$. Then,  $f(\hat{x};{\psi^*},\mathbf{e}) <  f(\hat{x};1,\mathbf{e}) \leq f(\tilde{x};1,\mathbf{e})$,
     where the first inequality follows from Lemma \ref{lem:nlp_aug_psi_decrease} and the second from the optimality of $\tilde{x}$.
     This completes the proof.
\qed\end{proof}


\section{\texorpdfstring{Computing the g-scaling parameter $\Phi$}{Computing the g-scaling parameter Phi}}\label{sec:diag-scale}

In this section, we consider $\psi>0$ fixed,  
and we focus on a good selection of the parameter $\Phi$ for the \ref{hNLP} bound for \ref{MEDP}. The inspiration for the parameterization of the bound by a scaling vector is taken from the g-scaling concept of \citep*{gscale}.
Because of this, we refer to $\Phi$ as the \emph{g-scaling} parameter, and choosing $\Phi$ different from $\mathbf{e}$ as \emph{g-scaling}.
Applying g-scaling preserves the semidefiniteness condition required to ensure  convexity of the relaxation; that is, if  $C_2\succeq \psi C_1 $ then  $\Diag(\Phi) C_2 \Diag(\Phi)\succeq \psi \Diag(\Phi) {C}_1 \Diag(\Phi)$.
 The challenge is to find a good g-scaling for reducing the \ref{hNLP} bound.

    In what follows, we consider $E_1 := \psi \Diag(\Phi) {C}_1 \Diag(\Phi)$ and $E_2 := \Diag(\Phi) C_2 \Diag(\Phi)$, and we apply the theory presented in \S\S \ref{sec:compare-mersp-bounds}-\ref{sec:aug-nlp-bound} with $\psi {C}_1$ and $C_2$\,, respectively replaced by $E_1$ and $E_2$\,.


\subsection{A selection strategy for the g-scaling parameter }\label{subsec:init-psi}

 We propose a strategy for selecting the g-scaling parameter $\Phi$, using the following  optimization problem, which aims to reduce the maximum eigenvalue of $C_2$\,.
\begin{equation}\label{eq:min_maxeig_C2_subprob}
\min\left\{\tau \,:\,\tau I_n - \Diag(\Phi)C_2\Diag(\Phi)\succeq 0,\,\left({\textstyle\prod_{i=1}^n \Phi_{i}}
\right)^{1/n}\geq 1,\, \Phi \geq \mathbf{0}\right\}.
    \end{equation}
In preliminary experiments, we found that this strategy worked quite well, compared to others that we tried. A possible explanation for this
good behavior is as follows. 
The idea is to choose $\Phi$ so that $\Diag(\Phi)C_2\Diag(\Phi)\in\mathbb{S}^n_{+}$ has a compressed range of nonzero eigenvalues. We note that the geometric-mean constraint is a normalization, based on the simple observation that 
the \ref{hNLP} bound is invariant under positive scaling of $\Phi$; it serves to
guarantee that $\Phi\in\mathbb{R}^n_{++}$\,. Moreover,   $\det(\Diag(\Phi)C_2\Diag(\Phi))\geq \det(C_2)$,  with equality if $\Phi$ is an optimal solution to \eqref{eq:min_maxeig_C2_subprob}, as we can see from the proof of Theorem \ref{thm:nlp_id_eq_nlp_tr}. So, for the case of $C_2\succ 0$, we minimize the largest eigenvalue of $\Diag(\Phi)C_2\Diag(\Phi)$, while keeping its determinant bounded from below by a positive number.  
By eigenvalue-interlacing inequalities, compressing the range of eigenvalues of $\Diag(\Phi)C_2\Diag(\Phi)$ serves to limit the
range of $ \ldet (\Diag(\Phi)C_2\Diag(\Phi))[S(x),S(x)]$, as $x$
varies over the feasible region of \ref{MEDP}.
In this way, the associated  
g-scaled \ref{MEDP} instance becomes closer to a \ref{MESP} instance. Our experience is that \ref{MESP} instances tend to have smaller gaps than comparable \ref{MERSP} instances.

\begin{proposition}
Let $C_2 \succeq 0$. Then,  \eqref{eq:min_maxeig_C2_subprob} is a convex optimization problem satisfying Slater's condition.
\end{proposition}
\begin{proof}
The 
 geometric mean is of course concave.
Because $C_2\succeq 0$, we have that the mapping $(\tau,\Phi)\rightarrow \tau  I_n - \Diag(\Phi)C_2\Diag(\Phi)$ is matrix-concave in $(\tau,\Phi)$. 
It follows that \eqref{eq:min_maxeig_C2_subprob} is a convex optimization problem.
To verify Slater's condition, note that $(\hat\tau,\hat\Phi):=(5\lambda_1(C_2),2\mathbf{e})$ is strictly feasible for \eqref{eq:min_maxeig_C2_subprob}.\qed
\end{proof}

In the next theorem, we highlight an important property of this strategy to select the g-scaling parameter $\Phi$. Specifically, we demonstrate that when $\Phi$ is obtained as the solution of \eqref{eq:min_maxeig_C2_subprob}, the 
$\TraceNLP$ strategy for selecting parameters $\gamma$ and $d$ for the \ref{hNLP} bound leads to the same choice as the $\IdentityNLP$ strategy. Furthermore, for the $\IdentityNLP$ strategy, the procedure described in the next section can be applied to refine the g-scaling parameter $\Phi$, potentially improving the resulting bound. 

\begin{theorem}\label{thm:nlp_id_eq_nlp_tr}
    Let $\hat\Phi$ be an optimal solution of \eqref{eq:min_maxeig_C2_subprob}, and define $E_2:=\Diag(\hat\Phi)C_2\Diag(\hat\Phi)$. Then $D := \lambda_{1}(E_2)I_n$  is the unique optimal solution of the trace-minimization problem in \eqref{eq:nlp-trace}, with $C_2$ replaced by  $E_2$\,. Consequently,  {\rm hNLP-Tr} is precisely {\rm hNLP-Id}. 
\end{theorem}

\begin{proof} 
    The optimality conditions of \eqref{eq:min_maxeig_C2_subprob} are:
    \begin{equation}\label{eq:opt_conditions_min_max_eigval_C2}
        \begin{array}{lll}
    &\tau I_n - \Diag(\Phi)C_2\Diag(\Phi)\succeq 0,\quad &\left({\prod_{i=1}^n \Phi_{i}}
\right)^{1/n}\geq 1, \\[3pt]
    &\Tr(Z(\tau I_n - \Diag(\Phi)C_2\Diag(\Phi))) =0,\quad &\Tr(Z) = 1,
\\[3pt]
&2(Z\Diag(\Phi)C_2)_{ii} = \frac{\nu}{n\Phi_i} \left({\prod_{j=1}^n \Phi_{j}}
\right)^{1/n} \!+ \upsilon_i\,,\;\, & \upsilon_i \Phi_i = 0,~ i\in N,\\[3pt]
&\Phi \geq \mathbf{0},\;\upsilon \geq \mathbf{0}, \;\nu \geq 0,\; Z \succeq 0, \quad &\nu\left(1-\left({\prod_{i=1}^n \Phi_{i}}
\right)^{1/n}\right) = 0.
\end{array}
    \end{equation}
  Let  $(\hat \tau, \hat\Phi, \hat \upsilon, \hat \nu, \hat Z)$ satisfy \eqref{eq:opt_conditions_min_max_eigval_C2}. Let $E_2:=\Diag(\hat{\Phi} )C_2\Diag(\hat{\Phi} )$ and $W:=\hat{\tau}  I_n-E_2$\,. 
   
   As $\left({\prod_{i=1}^n \hat{\Phi} _{i}}
\right)^{1/n}\geq 1$, we have  that $\hat{\Phi} \in \mathbb{R}^n_{++}$\,. Therefore, $\hat{\upsilon}  = \mathbf{0}$ by complementary slackness. 
Because $W\succeq 0$\,, it follows that  $\hat{\tau}  \geq \lambda_1(E_2) > 0$.
Because $\hat{Z} \succeq 0$, $W\succeq 0$ and $\Tr(\hat{Z} W) =0$, it follows that  $\hat{Z} W = W\hat{Z}  = 0$. Hence,
    \begin{equation}\label{eq:lambda_max_Z_min_maxlmax}
          E_2\hat{Z}  = \hat{\tau}  \hat{Z} ,\text{ and } \hat{Z} \Diag(\hat{\Phi} )C_2 = \hat{\tau}  \hat{Z} \Diag(\hat{\Phi} )^{-1},
    \end{equation}
    and, therefore, we see that $\hat{\tau} $ is an eigenvalue of $E_2$ (noting that $\Tr(\hat{Z} )=1$, and thus,  $\hat{Z} \neq0$).  Consequently, $\hat{\tau} \leq\lambda_1(E_2)$. Together with $\hat{\tau} \geq\lambda_1(E_2)$, we conclude that $\hat{\tau} =\lambda_1(E_2)$, hence $E_2\hat{Z} =\lambda_1(E_2)\hat{Z} $.
    
    From 
    $\hat{\upsilon} =\mathbf{0}$ and the second identity in \eqref{eq:lambda_max_Z_min_maxlmax},  we have that
    \begin{align*}
        &\left(\hat{Z} \Diag(\hat{\Phi} )C_2\right)_{ii} = \frac{\hat{\nu} }{2n\hat{\Phi} _i} \left({\textstyle\prod_{j=1}^n \hat{\Phi} _{j}}
\right)^{1/n}\Leftrightarrow\\&\hat{\tau}  \frac{\hat{Z} _{ii}}{\hat{\Phi} _i} = \frac{\hat{\nu} }{2n\hat{\Phi} _i} \left({\textstyle\prod_{j=1}^n \hat{\Phi} _{j}}
\right)^{1/n} \Leftrightarrow \hat{Z} _{ii} = \frac{\hat{\nu} }{2n\hat{\tau}  } \left({\textstyle\prod_{j=1}^n \hat{\Phi} _{j}}
\right)^{1/n}\,,
    \end{align*}
for all $i\in N$. Then, because $\Tr(\hat{Z} ) = 1$, we have $\hat{\nu}  = 2\hat{\tau}  /\big({\textstyle\prod_{j=1}^n \hat{\Phi} _{j}}
\big)^{1/n}$. As $\hat{\tau}  > 0$, we have that $\hat{\nu}  > 0$. Then, from complementary slackness, we have that $\textstyle\prod_{j=1}^n \hat{\Phi} _{j} = 1$. We conclude that $\diag(\hat{Z} ) = (1/n)\mathbf{e}$.
    
Throughout  the remainder of this proof,  whenever we refer to \eqref{eq:nlp-trace}, we consider that  $C_2$ is replaced by  $E_2$\,.  Under this convention, the dual of \eqref{eq:nlp-trace} is
    \begin{equation}\label{eq:nlp-trace-dual-aug-scaling}
    \max\,\{\Tr(E_2 \Omega )\,:\, \diag(\Omega)=\mathbf{e},\, \Omega\succeq 0\}\,.
\end{equation}
As $\diag(\hat{Z} )=(1/n)\mathbf{e}$ and $\hat{Z} \succeq0$, the matrix $\hat\Omega := n\hat{Z} $ satisfies $\diag(\hat\Omega)=\mathbf{e}$ and $\hat\Omega\succeq0$, so $\hat\Omega\succeq0$ is feasible for \eqref{eq:nlp-trace-dual-aug-scaling}, with objective value  
\[
\Tr(E_2 \hat\Omega) = n\Tr(E_2\hat{Z} ) = n\hat{\tau}  \Tr(\hat{Z} ) = n\lambda_{1}(E_2),
\]
where the second equation comes from \eqref{eq:lambda_max_Z_min_maxlmax}. The matrix $\hat Y := \lambda_{1}(E_2)I_n$ is diagonal and satisfies $\hat Y\succeq E_2$\,, so it is feasible for \eqref{eq:nlp-trace}, with objective value $\Tr(\hat Y) = n\lambda_{1}(E_2) = \Tr(E_2\hat\Omega)$. By weak duality, $\hat Y$ is optimal. 

Next, we will demonstrate that $\hat{Y}$ is the unique optimal solution for \eqref{eq:nlp-trace}. 
    Consider any optimal solution $Y$ to \eqref{eq:nlp-trace}, then $\Tr(Y) = n\lambda_{1}(E_2)$. As $E_2\hat{Z}  = \lambda_{1}(E_2)\hat{Z} $, we have $E_2 \hat\Omega = \lambda_{1}(E_2)\hat\Omega$, so
    \[
    \Tr((Y-E_2)\hat\Omega) = \Tr(Y\hat\Omega) -  \Tr(E_2\hat\Omega) = \textstyle\sum_{i = 1}^n Y_{ii} \hat\Omega_{ii} - n\lambda_{1}(E_2) =   \textstyle\sum_{i = 1}^n Y_{ii} - n\lambda_{1}(E_2) = 0.
    \]
Because $Y-E_2 \succeq 0$, $\hat\Omega \succeq 0$ and $\Tr((Y-E_2)\hat\Omega) = 0$, we have $(Y-E_2)\hat\Omega = 0$, hence $(Y-\lambda_{1}(E_2)I_n)\hat\Omega = 0$. For each $i \in N$ this gives $(Y_{ii}-\lambda_{1}(E_2))\hat\Omega_{ii}= 0$, and as $\diag(\hat\Omega) = \mathbf{e}$, $Y_{ii} = \lambda_{1}(E_2)$. Therefore, $\hat Y$ is the unique optimal solution to  \eqref{eq:nlp-trace}.
\qed\end{proof}


\subsection{Optimizing the g-scaling parameter}\label{subsec:optimize-Psi}

Next, we 
describe an optimization procedure for computing a g-scaling parameter $\Phi$ that locally minimizes   the \ref{hNLP} bound  corresponding to the choice of parameters prescribed by the 
$\IdentityNLP$  strategy, that is, the {\rm hNLP-Id} bound, with $\psi C_1 $ replaced by  $E_1 := \psi \Diag(\Phi){C}_1\Diag(\Phi)$ and ${C}_2$ replaced by  $E_2 := \Diag(\Phi){C}_2\Diag(\Phi)$. This procedure was motivated by a similar  approach used in  \citep*[Section 5]{gscale} for the g-scaling of various bounds for \ref{MESP}.

The algorithm is initialized with a g-scaling vector $\Phi^0$ obtained as the solution of \eqref{eq:min_maxeig_C2_subprob}.

The optimization of $\Phi$ in {\rm hNLP-Id} relies on  two key properties. First,  $p$ and $\gamma d$ are fixed and do not depend on $\Phi$. Second, for any fixed $\Phi$, the corresponding ordinary scaling parameter $\gamma$ is uniquely determined. As a result, the objective of {\rm hNLP-Id} can be regarded as a function of $\Phi$ alone. 
This feature enables the use of the BFGS method to locally minimize {\rm hNLP-Id} with respect to $\log\Phi$, adapting the methodology from \citep*[Section 5]{gscale}. Specifically, for a given $\Phi$, let $\hat{x} := \hat{x}(\Phi)$ denote an optimal solution of {\rm hNLP-Id} corresponding to this choice of $\Phi$. For $k=1,2$, 
we define
    \[ \textstyle
    h^k(\Phi) :=  \ldet\left( I_n + \Diag(\hat{x})^{\scriptscriptstyle 1/2}\left( \,\frac{1}{\lambda_{1}(E_2(\Phi))}\,E_k(\Phi) -I_n\right)\Diag(\hat{x})^{\scriptscriptstyle 1/2}\right).
    \]
    Then, the {\rm hNLP-Id} objective function can be written as 
    $
    \phi(\Phi) := h^1(\Phi)-h^2(\Phi).
    $
    
    The optimization procedure proceeds iteratively. At each iteration, for the current value of $\Phi$, the vector $\hat{x}:=\hat{x}(\Phi)$ is first obtained by solving {\rm hNLP-Id} for the given $\Phi$. The function $\phi(\Phi)$ is then locally minimized with respect to $\log\Phi$, treating $\hat{x}$ as fixed when computing a subgradient and performing the BFGS update. Although $\hat{x}$ is recomputed at every iteration as a function of $\Phi$, it is treated as constant during the update of $\Phi$. This separation is essential to ensure that, at every iteration of the algorithm, the objective value of $\phi(\Phi)$ remains an upper bound for \ref{MERSP}, and it follows the methodology of \citep*[Section 5]{gscale}.


\section{Branch-and-bound}\label{sec:BB}

A main use of upper bounds for \ref{MERSP}
is within a B\&B algorithm. In this section, we describe two techniques that can enhance the
performance of B\&B. These two concepts 
inform
our implementation. 
One technique, described in Section \ref{sec:varfix}, is 
variable fixing, which is often 
applied for mixed-integer linear and mixed-integer convex optimization, using duality information 
to fix integer variables at one of their bounds.
This can be applied on every B\&B subproblem. 
The other technique, described in Section \ref{sec:subprob}, which at first blush 
appears to be trivial, involves exactly how to construct child subproblems for \ref{MERSP}, when branching in B\&B. 


\subsection{Variable fixing based on duality}\label{sec:varfix}

In what follows, we discuss the application of a classical variable-fixing methodology to \ref{MERSP} based on  knowledge of a lower bound for its objective value and a feasible solution to the Lagrangian dual of \ref{hNLP} (see, e.g., \citep*{yamagishi2026dualpathfixingstrategyapplication}) .

The Lagrangian dual of {\rm\ref{hNLP}} can be formulated as 
\begin{equation}\label{eq:dual_nlp}\tag{DhNLP}
\begin{array}{rrl}
&\min &
f(x;\psi,\Phi) + \upsilon^\top x + \nu^\top(\mathbf{e}-x) + \tau(s-\mathbf{e}^\top x)\\
&\text{\rm s.t.} 
&\nabla_x f(x;\psi,\Phi) + \upsilon - \nu - \tau\mathbf{e}  = \mathbf{0};\\[2.5pt]
&&\upsilon \geq \mathbf{0}, \nu \geq \mathbf{0}.
\end{array}
\end{equation}
Fixing $x$ in \ref{eq:dual_nlp} at given feasible solution $\hat{x}$ of \ref{hNLP}, we show how to construct a closed-form feasible solution of \ref{eq:dual_nlp},  with the goal of obtaining a small duality gap. The minimum gap between $\hat x$ in \ref{hNLP} and feasible solutions  of  \ref{eq:dual_nlp} of the form $(\hat x,\upsilon,\nu,\tau)$,   is given by the optimal value of the linear-optimization problem
\begin{equation}\label{eq:g_theta}
\begin{array}{rrl}
&\min & \nu^\top \mathbf{e} + \tau s - \nabla_x f(\hat{x};\psi,\Phi)^\top\hat{x}\\
&\text{s.t.} 
& \nabla_x f(\hat{x};\psi,\Phi) + \upsilon - \nu - \tau\mathbf{e} = \mathbf{0},\\
&&\nu \geq  \mathbf{0}, \upsilon \geq  \mathbf{0}.
\end{array}
\end{equation}
The  optimality conditions for  \eqref{eq:g_theta} are 
\begin{equation}\label{kktnatural}
\begin{array}{l}
    \nabla_x f(\hat{x};\psi,\Phi) + \upsilon - \nu - \tau\mathbf{e}  = \mathbf{0},~\nu\geq  \mathbf{0},~\upsilon\geq \mathbf{0},\\
    \mathbf{e}^\top x = s,x\in[0,1]^n,\\
     \nu^\top \mathbf{e} + \tau s = \nabla_x f(\hat{x};\psi,\Phi)^\top x.
\end{array}
\end{equation}
Let $\sigma$ be a permutation of $N$, such that $\nabla_x f(\hat{x};\psi,\Phi)_{\sigma(1)} \geq \dots \geq \nabla_x f(\hat{x};\psi,\Phi)_{\sigma(n)}$\,. 
We can verify that the following solution satisfies \eqref{kktnatural}.

\begin{align}
    &\tilde{x}_{\sigma(\ell)}:= \begin{cases}
        1,\quad&\text{for }\ell =1,\dots,s;\\
        0,\quad&\text{otherwise},
    \end{cases}\nonumber\\ 
    &\tilde{\tau} :=  \nabla_x f(\hat{x};\psi,\Phi)_{\sigma(s)}\,,\label{eq:tau}\\
    &\tilde{\nu}_{\sigma(\ell)} := \begin{cases}
         \nabla_x f(\hat{x};\psi,\Phi)_{\sigma(\ell)} - \tilde{\tau},\quad &\text{for }\ell = 1,\dots,s;\\
         0,&\text{otherwise},
    \end{cases}\label{eq:nu}\\
    &\tilde{\upsilon}_{\sigma(\ell)} := \begin{cases}
         \tilde{\tau}-\nabla_x f(\hat{x};\psi,\Phi)_{\sigma(\ell)}\,,\quad &\text{for }\ell = s+1,\dots,n;\\
         0,&\text{otherwise}.\label{eq:upsilon}
    \end{cases}
\end{align}

Finally, in the following theorem, we apply a classical  variable-fixing procedure to \ref{MERSP} using the dual solution constructed above. We note that the proposed construction of dual-feasible solutions is particularly useful when \ref{hNLP} is not solved to optimality, and only feasibility of the given point $\hat{x}$ can be guaranteed.  

\begin{theorem}\label{thm:variable_fixing} 
 Let 
\begin{itemize}
    \item\!${\rm{LB}}$ be the objective-function value of a feasible solution for {\rm\ref{MERSP}},
    \item\!$\hat{x}$ be a feasible solution for {\rm\ref{hNLP}},
    \item $(\hat{x},\tilde\upsilon,\tilde\nu,\tilde\tau)$ be the feasible solution for {\rm\ref{eq:dual_nlp}} given by (\ref{eq:tau}--\ref{eq:upsilon}),
    with objective-function value $\hat{\zeta}$.
\end{itemize}
Then, for every optimal solution $x^*$ 
for {\rm\ref{MERSP}}, we have:
\[
\begin{array}{ll}
x_j^*=0, ~ \forall\,j\in N \mbox{ such that } \hat{\zeta}-{\rm LB}<  \tilde{\upsilon}_j\,,\\
x_j^*=1, ~ \forall\,j\in N \mbox{ such that } 
\hat{\zeta}-{\rm LB} < \tilde{\nu}_j\,.
\end{array}
\]
\end{theorem}


\subsection{Constructing subproblems}\label{sec:subprob}

In this section, we investigate down- and up-branching strategies for \ref{MERSP} B\&B subproblems. 
For a given variable $x_j$\,, we consider two branching strategies for \ref{MERSP}, described in terms of  the more-general problem \ref{MEDP}:
\begin{itemize}
    \item \hypertarget{Sfixtarget}{$\Sfix$} (\emph{constrained}): creates two child subproblems by appending the constraint $x_j=0$ or $x_j=1$ to each subproblem created, and keeps $C_1$ and $C_2$ unchanged;
    \item \hypertarget{Rfixtarget}{$\Rfix$} (\emph{reduction}): creates two child subproblems in which  $C_1$ and $C_2$ are replaced by matrices   of order $n-1$. Specifically, for $k=1,2$:  
    \begin{itemize}
    \item in the child subproblem corresponding to  $x_j=0$, $C_k$ is replaced by the principal submatrix $\tilde C_k^0 := C_k[N\setminus j,N\setminus j]$;
    \item in the child subproblem corresponding to $x_j=1$, $C_k$ is replaced by  the Schur complement 
    \[
    \tilde C_k^1 := C_k/C_k[j,j] = C_k[N\setminus j, N\setminus j] - \frac{1}{C_k[j,j]}\,C_k[N\setminus j, j]\,C_k[j,N\setminus j],
    \]
    and the constant $\log(C_1[j,j]/C_2[j,j])$ is added to the objective.
\end{itemize}
\end{itemize}

We begin by observing that matrix reduction preserves the Loewner order. 
 
\begin{proposition}\label{prop:psd_preserved}
Let $C_2 \succeq C_1 \succeq 0$. Let  $(\tilde C_1^0, \tilde C_2^0)$ and $(\tilde C_1^1, \tilde C_2^1)$ be the matrices obtained from $(C_1,C_2)$ by the $\Rfix$ strategy, in the subproblems corresponding to $x_j=0$ and $x_j=1$, respectively. Then $\tilde C_2^0 \succeq \tilde C_1^0 \succeq 0$ and $\tilde C_2^1 \succeq \tilde C_1^1 \succeq 0$.
\end{proposition}
\begin{proof}
For the subproblem corresponding to $x_j=0$, the result follows immediately from the fact that $C_2 \succeq C_1$ implies $C_2[N\setminus j,N\setminus j] \succeq C_1[N\setminus j,N\setminus j]$. 
For the subproblem corresponding to $x_j=1$, from \citep*[Theorem 3.1(b)]{li2000extremal} we have that $C_2/C_2[j,j] \succeq C_1/C_1[j,j]$.
\qed\end{proof}

Assuming $C_2 \succ 0$, we now investigate the behavior of the scaling $\psi$ in \ref{hNLP} under the two branching strategies. By Theorem \ref{thm:best_alpha_NLPId}, the bound is nonincreasing in $\psi$, so it is natural to work with the largest admissible value, which we denote $\psi^*$. The $\Sfix$ strategy leaves the data matrices unchanged, so $\psi^*$ is constant across subproblems. The $\Rfix$ strategy, on the other hand, modifies the matrices; we demonstrate below that $\psi^*$ cannot decrease under this modification, and may in fact increase. 
\begin{lemma}\label{lem:Rfix_better_psi_over_Ffix}
Let $C_2 \succeq C_1 \succeq 0$ with $C_2 \succ 0$. Let  $(\tilde C_1^0, \tilde C_2^0)$ and $(\tilde C_1^1, \tilde C_2^1)$ be the matrices obtained from $(C_1,C_2)$ by the $\Rfix$ branching strategy, in the subproblems corresponding to $x_j=0$ and $x_j=1$, respectively.

Let
\[
\psi^* := 1 / \lambda_1(C_2^{-1/2} C_1 C_2^{-1/2}), \quad
\tilde\psi^*_0 := 1 / \lambda_1((\tilde C_2^0)^{-1/2} \tilde C_1^0 (\tilde C_2^0)^{-1/2}), \quad
\tilde\psi^*_1 := 1 / \lambda_1((\tilde C_2^1)^{-1/2} \tilde C_1^1 (\tilde C_2^1)^{-1/2}),
\]
which are the best 
scalings for the original problem and the subproblems corresponding to $x_j=0$ and $x_j=1$, respectively (see Theorem \ref{thm:best_alpha_NLPId}). 
Then, we have $\tilde\psi^*_0 \geq \psi^*$ and $\tilde\psi^*_1 \geq \psi^*$.
\end{lemma}
\begin{proof}
We show that $\tilde C_2^\ell \succeq \psi^* \tilde C_1^\ell$, for $\ell=0,1$.
\begin{itemize}
    \item For the subproblem corresponding to $x_j=0$: $C_2 \succeq \psi^* C_1$ implies $C_2[N \setminus j, N \setminus j] \succeq \psi^* C_1[N \setminus j, N \setminus j]$, so $\tilde C_2^0 \succeq \psi^* \tilde C_1^0$\,.
    \item For the subproblem corresponding to $x_j=1$: By \citep*[Theorem 3.1(b)]{li2000extremal}, $C_2 \succeq \psi^* C_1$ gives $C_2 / C_2[j,j] \succeq (\psi^* C_1) / (\psi^* C_1)[j,j]$. Because $(\psi^* C_1)/(\psi^* C_1)[j,j] = \psi^*(C_1/C_1[j,j])$, we obtain $\tilde C_2^1 \succeq \psi^* \tilde C_1^1$\,.
\end{itemize}
Then, for $\ell=0,1$, we have $I_{n-1} \succeq \psi^*\, (\tilde C_2^\ell)^{-1/2} \tilde C_1^\ell (\tilde C_2^\ell)^{-1/2}$, and hence $\lambda_1((\tilde C_2^\ell)^{-1/2} \tilde C_1^\ell (\tilde C_2^\ell)^{-1/2}) \leq (\psi^*)^{-1}$, and therefore $\tilde\psi^*_\ell \geq \psi^*$.
\qed\end{proof}

For hNLP-Id, we show that the bound obtained by applying the branching strategy $\Rfix$ is at least as strong as the bound obtained by applying $\Sfix$. Our analysis relies on the  monotonicity of the  objective function of hNLP-Id with respect to the scaling parameters $\psi$ and $\gamma$. Analogous comparisons of branching strategies within B\&B frameworks  for the related  problems \ref{MESP} and another generalized version of \ref{MESP} can be found in \citep*[Sec. 4]{MESP2DOPT} and \citep*[Sec. 7]{ExtendedGMESP}. 

\begin{lemma}\label{lem:hNLPId_nonincreasing}
    Let $C_1,C_2\in\mathbb{S}^n_+$\,, $\psi>0$, $\Phi\in\mathbb{R}^n_{++}$\,, and $\hat{x} \in [0,1]^n$ with  $\mathbf{e}^\top \hat{x} = s$. Assume that $C_2 \succeq  \psi C_1 \succeq 0$.  Define $\hat{C}_1:=\psi C_1$\,, $\hat{C}_2:= C_2$\,. For $k=1,2$ and $\gamma>0$, let
    \[
L_k(\gamma) := I_n + \Diag(\hat{x})^{\scriptscriptstyle 1/2}(\gamma \Diag(\Phi)\hat{C}_k\Diag(\Phi) - I_n)\Diag(\hat{x})^{\scriptscriptstyle 1/2},
\]
and assume that $L_k(\gamma)\succ 0$  \, for all  $\gamma>0$.  
Let 
$h(\gamma) := \ldet L_1(\gamma) - \ldet L_2(\gamma) - s\log\psi$. 
Then, $h(\gamma)$ is nonincreasing on $\mathbb{R}_{++}$\,.
\end{lemma}

\begin{proof}
    For $k=1,2$, we have 
\[
\frac{\partial \ldet (L_k)}{\partial \gamma} = \tfrac{1}{\gamma}\bigl(n - \Tr\bigl(L_k(\gamma)^{-1}\Diag(\mathbf{e}-\hat{x})\bigr)\bigr),
\]
and therefore
\begin{align*}
h'(\gamma) 
&= \tfrac{1}{\gamma}\Bigl(\Tr\bigl(L_2(\gamma)^{-1}\Diag(\mathbf{e}-\hat{x})\bigr) - \Tr\bigl(L_1(\gamma)^{-1}\Diag(\mathbf{e}-\hat{x})\bigr)\Bigr) \\
&= \tfrac{1}{\gamma}\,\Tr\Bigl(\Diag(\mathbf{e}-\hat{x})^{\scriptscriptstyle 1/2}\bigl(L_2(\gamma)^{-1}-L_1(\gamma)^{-1}\bigr)\Diag(\mathbf{e}-\hat{x})^{\scriptscriptstyle 1/2}\Bigr).
\end{align*}
As $\hat C_2 \succeq \hat C_1$\,, we have 
\begin{align*}
    &\Diag(\Phi)\hat C_2\Diag(\Phi) \succeq \Diag(\Phi)\hat C_1\Diag(\Phi) \Rightarrow \\
    &\gamma \Diag(\Phi)\hat C_2\Diag(\Phi) -I_n\succeq \gamma \Diag(\Phi)\hat C_1\Diag(\Phi)-I_n \Rightarrow\\
    &\Diag(\hat{x})^{1/2}(\gamma \Diag(\Phi)\hat  C_2\Diag(\Phi)-I_n)\Diag(\hat{x})^{1/2}\succeq \Diag(\hat{x})^{1/2}(\gamma \Diag(\Phi)\hat C_1\Diag(\Phi)-I_n)\Diag(\hat{x})^{1/2}\Rightarrow\\
    &L_2(\gamma)\succeq L_1(\gamma) ~\Rightarrow~ L_1(\gamma)^{-1}\succeq L_2(\gamma)^{-1}.
\end{align*}
The result follows.\qed
\end{proof}

\begin{theorem}\label{thm:Rfix_better_NLPId_over_Ffix}
Suppose that $C_2 \succeq C_1 \succeq 0$ and $C_2\succ 0$. Then, for the \rm{hNLP-Id} relaxation, the bound obtained under the $\Rfix$ branching strategy is at least as strong as the bound obtained under the  $\Sfix$ branching strategy. 
\end{theorem}
\begin{proof}
Fix any $\hat{x} \in [0,1]^n$ satisfying $\mathbf{e}^\top \hat{x} = s$. Let $x_j$ be the branching variable and $\hat{\Phi} \in \mathbb{R}^n_{++}$\, be arbitrary. We consider the following definitions.
\begin{itemize}
    \item[For $\Sfix$:] 
    \[
    \begin{array}{l}
         L_1(x;\gamma, \psi,\Phi) := I_n + \Diag({x})^{\scriptscriptstyle 1/2}(\gamma \psi \Diag(\Phi)C_1\Diag(\Phi) - I_n)\Diag({x})^{\scriptscriptstyle 1/2},\\[4pt]
         L_2(x;\gamma,\Phi) := I_n + \Diag({x})^{\scriptscriptstyle 1/2}(\gamma \Diag(\Phi)C_2\Diag(\Phi) - I_n)\Diag({x})^{\scriptscriptstyle 1/2},
    \end{array}
    \]
    and
    \[
    h(x;\gamma, \psi,\Phi) := \ldet L_1(x;\gamma, \psi,\Phi) - \ldet L_2(x;\gamma,\Phi) - s\log\psi,
    \]
    for $x\in\mathbb{R}^n$, $\gamma>0$, $\psi>0$, and $\Phi\in\mathbb{R}^n_{++}$\,.
    \item[ For $\Rfix$:]   For $\ell=0,1$,  
    \[
    \begin{array}{l}
         \tilde{L}_1^\ell(x;\gamma, \psi,\Phi) := I_{n-1} + \Diag({x})^{\scriptscriptstyle 1/2}(\gamma \psi \Diag(\Phi)\tilde{C}_1^\ell\Diag(\Phi) - I_{n-1})\Diag({x})^{\scriptscriptstyle 1/2},\\[4pt]
         \tilde{L}_2^\ell(x;\gamma,\Phi) := I_{n-1} + \Diag({x})^{\scriptscriptstyle 1/2}(\gamma \Diag(\Phi)\tilde{C}_2^\ell\Diag(\Phi) - I_{n-1})\Diag({x})^{\scriptscriptstyle 1/2},
    \end{array}
    \]
    and
    \[
    \begin{array}{l}
    \tilde{h}^0(x;\gamma, \psi,\Phi) := \ldet \tilde{L}_1^0(x;\gamma, \psi,\Phi) - \ldet \tilde{L}_2^0(x;\gamma,\Phi)- s\log\psi,\\[4pt]
    \tilde{h}^1(x;\gamma, \psi,\Phi) := \ldet \tilde{L}_1^1(x;\gamma, \psi,\Phi) - \ldet \tilde{L}_2^1(x;\gamma,\Phi)- (s-1)\log\psi + \log(C_1[j,j]/C_2[j,j]),
    \end{array}
    \]
    for $x\in\mathbb{R}^{n-1}$, $\gamma>0$, $\psi>0$, and $\Phi\in\mathbb{R}^{n-1}_{++}$\,.
\end{itemize}
 Next, consider $\tilde{x}:=\hat{x}[N\setminus j]$, $\tilde{\Phi}:=\hat{\Phi}[N\setminus j]$, and
\[
\begin{array}{l}
     M := I_{n-1} + \Diag(\tilde{x})^{\scriptscriptstyle 1/2}(\gamma \psi \Diag(\tilde{\Phi})C_1[N\setminus j, N\setminus j]\Diag(\tilde{\Phi}) - I_{n-1})\Diag(\tilde{x})^{\scriptscriptstyle 1/2},\\[4pt]
     w := \hat{x}_j^{\scriptscriptstyle 1/2}\gamma \psi \hat{\Phi}_j \Diag(\tilde{x})^{\scriptscriptstyle 1/2}\Diag(\tilde{\Phi}) C_1[N\setminus j, j],
     \\[4pt]
     y := 1 + \hat{x}_j\big(\gamma \psi \hat{\Phi}_j^2 C_1[j,j] - 1\big).
\end{array}
\]
Note, from the block structure of $\Diag(\hat{\Phi})C_1\Diag(\hat{\Phi})$, that
\[
L_1(\hat{x};\gamma, \psi,\hat{\Phi}) = \left(\begin{smallmatrix} M & w \\ w^\top & y \end{smallmatrix}\right). 
\]
Then,
\begin{itemize}
    \item if $\hat{x}_j = 0$, then $w = \mathbf{0}$ and $y = 1$. As  $\tilde{C}_1^0 = C_1[N\setminus j, N\setminus j]$, we have  $M = \tilde{L}_1^0(\tilde{x};\gamma, \psi,\tilde{\Phi})$. Therefore, $L_1(\hat{x};\gamma, \psi,\hat{\Phi}) = \left(\begin{smallmatrix} \tilde{L}_1^0(\tilde{x};\ \gamma,  \psi,\tilde{\Phi}) & \mathbf{0} \\ \mathbf{0}^\top & 1 \end{smallmatrix}\right)$, which implies $\ldet L_1(\hat{x};\gamma, \psi,\hat{\Phi}) = \ldet \tilde{L}_1^0(\tilde{x};\gamma, \psi,\tilde{\Phi})$. With an analogous argument, we can verify that $\ldet L_2(\hat{x};\gamma,\hat{\Phi}) = \ldet \tilde{L}_2^0(\tilde{x};\gamma,\tilde{\Phi})$. Hence, 
    \begin{equation}\label{eq:NLPId_StoR_0}
    h(\hat{x};\gamma, \psi,\hat{\Phi}) = \tilde{h}^0(\tilde{x};\gamma, \psi,\tilde{\Phi}).
    \end{equation}
    
    \item if $\hat{x}_j = 1$, then $y = \gamma \psi \hat{\Phi}_j^2 C_1[j,j]$ and $w = \gamma \psi \hat{\Phi}_j \Diag(\tilde{x})^{\scriptscriptstyle 1/2}\Diag(\tilde{\Phi}) C_1[N\setminus j, j]$. Taking the Schur complement of the scalar block $y$\, in $L_1(\hat{x};\gamma,\psi,\hat{\Phi})$ yields 
    \[
    \ldet L_1(\hat{x};\gamma, \psi,\hat{\Phi}) = \ldet\big(M - \tfrac{1}{y} w w^\top\big) + \log y\,.
    \]
    As
    \[
    \tfrac{1}{y} w w^\top = \tfrac{\gamma \psi}{C_1[j,j]} \Diag(\tilde{x})^{\scriptscriptstyle 1/2}\Diag(\tilde{\Phi}) C_1[N\setminus j, j] C_1[j, N\setminus j] \Diag(\tilde{\Phi}) \Diag(\tilde{x})^{\scriptscriptstyle 1/2},
    \]
   we have
    \[
    M - \tfrac{1}{y} w w^\top = I_{n-1} + \Diag(\tilde{x})^{\scriptscriptstyle 1/2}(\gamma \psi \Diag(\tilde{\Phi})\tilde{C}_1^1\Diag(\tilde{\Phi}) - I_{n-1})\Diag(\tilde{x})^{\scriptscriptstyle 1/2} = \tilde{L}_1^1(\tilde{x};\gamma, \psi,\tilde{\Phi}).
    \]
    Hence, $\ldet L_1(\hat{x};\gamma, \psi,\hat{\Phi}) = \ldet \tilde{L}_1^1(\tilde{x};\gamma, \psi,\tilde{\Phi}) + \log(\gamma \psi \hat{\Phi}_j^2 C_1[j,j])$. With an analogous argument, we can verify that  $\ldet L_2(\hat{x};\gamma,\hat{\Phi}) = \ldet \tilde{L}_2^1(\tilde{x};\gamma,\tilde{\Phi}) + \log(\gamma \hat{\Phi}_j^2 C_2[j,j])$. Hence,
    \begin{equation}\label{eq:NLPId_StoR_1}
    \begin{array}{l}
    h(\hat{x};\gamma, \psi,\hat{\Phi}) = \ldet L_1(\hat{x};\gamma, \psi,\hat{\Phi}) - \ldet L_2(\hat{x};\gamma,\hat{\Phi}) -s\log\psi\\
    \qquad=\ldet \tilde{L}_1^1(\tilde{x};\gamma, \psi,\tilde{\Phi}) - \ldet \tilde{L}_2^1(\tilde{x};\gamma,\tilde{\Phi}) -s\log\psi + \log\frac{\psi C_1[j,j]}{C_2[j,j]} = \tilde{h}^1(\tilde{x};\gamma, \psi, \tilde{\Phi}).
    \end{array}
    \end{equation}
\end{itemize}

 Let 
    \[
    \hat\gamma := 1/\lambda_1(\Diag(\hat\Phi) C_2\Diag(\hat\Phi)),\qquad \tilde\gamma_\ell:=1/\lambda_1(\Diag(\tilde\Phi)\tilde C_2^\ell\Diag(\tilde\Phi)),~ \ell=0,1.
    \]
    We note that $
\lambda_1(\Diag(\tilde\Phi)\tilde C_2^\ell\Diag(\tilde\Phi)) \leq \lambda_1(\Diag(\hat\Phi) C_2\Diag(\hat\Phi))$,
where the inequality follows from  \citep*[Theorem 4.3.8]{HJBook} for $\ell=0$, and from \citep*[Corollary 2.3]{SchurBook} for $\ell=1$. Then, we have that 
\begin{equation}\label{comp_gamma}
    \tilde\gamma_0 \geq \hat\gamma,~\quad\mbox{ and }\quad \tilde\gamma_1 \geq \hat\gamma.
\end{equation}

 Let
\[
\hat\psi^* := 1 / \lambda_1(C_2^{-1/2} C_1 C_2^{-1/2}), \qquad
\tilde\psi^*_\ell := 1 / \lambda_1((\tilde C_2^\ell)^{-1/2} \tilde C_1^\ell (\tilde C_2^\ell)^{-1/2}),~\ell=0,1,
\]
be optimal parameters $\psi$ for the subproblems derived from the $\Sfix$ branching strategy, and the $\Rfix$  branching strategy for $\ell=0,1$, respectively (note that the optimal $\psi$ is invariant under g-scaling; see Lemma \ref{lem:help_psi_eigvals} in the Appendix). Then, by Lemma \ref{lem:Rfix_better_psi_over_Ffix}, we have that 
\begin{equation}\label{comp_psi}
\tilde\psi^*_0 \geq \hat\psi^*\quad\mbox{ and }\quad \tilde\psi^*_1 \geq \hat\psi^*.
\end{equation}

We  conclude that, 
\begin{itemize}
    \item if $\hat{x}_j = 0$,
    \begin{equation}\label{for0}
     h(\hat{x};\hat\gamma, \hat\psi^*,\hat{\Phi}) = \tilde{h}^0(\tilde{x};\hat\gamma, \hat\psi^*,\tilde{\Phi}) \geq \tilde{h}^0(\tilde{x};\tilde\gamma_0, \hat\psi^*,\tilde{\Phi}) \geq \tilde{h}^0(\tilde{x};\tilde\gamma_0, \tilde\psi^*_0,\tilde{\Phi}),
    \end{equation}
    where the equality follows from \eqref{eq:NLPId_StoR_0}. The first inequality follows from Lemma \ref{lem:hNLPId_nonincreasing} and \eqref{comp_gamma}, while
    the last inequality follows from Lemma \ref{lem:nlp_aug_psi_decrease} and \eqref{comp_psi}.
    \item  if $\hat{x}_j = 1$,
    \begin{equation}\label{for1}
     h(\hat{x};\hat\gamma, \hat\psi^*,\hat{\Phi}) = \tilde{h}^1(\tilde{x};\hat\gamma, \hat\psi^*,\tilde{\Phi}) \geq \tilde{h}^1(\tilde{x};\tilde\gamma_1, \hat\psi^*,\tilde{\Phi}) \geq \tilde{h}^1(\tilde{x};\tilde\gamma_1, \tilde\psi^*_1,\tilde{\Phi}),
    \end{equation}
    where the equality follows from \eqref{eq:NLPId_StoR_1}. 
    The first inequality follows from Lemma \ref{lem:hNLPId_nonincreasing} and \eqref{comp_gamma}, while the last inequality follows from Lemma \ref{lem:nlp_aug_psi_decrease} and \eqref{comp_psi}.
\end{itemize}
Finally, let us consider the child subproblem derived from the $\Rfix$ branching strategy corresponding to $x_j=0$. Consider the hNLP-Id relaxation for this subproblem, where the scaling parameters are $\tilde{\psi}^*_0$ and  $\tilde{\Phi}$. Assume that $\tilde{x}$ is an optimal solution for this subproblem, and let $\hat{x}[N\setminus j]:=\tilde{x}$ and $\hat{x}_j:=0$.  Let $\hat{\Phi}$ be the best known g-scaling parameter for the hNLP-Id relaxation for the subproblem derived from the $\Sfix$ branching strategy corresponding to $x_j=0$, and assume that $\tilde{\Phi}=\hat{\Phi}[N\setminus j]$. 
Then, considering \eqref{for0}, we have 
\[
\znlp(\tilde{C}^0_1,\tilde{C}^0_2,s;\tilde\psi_0^*,\tilde{\Phi}) = \tilde{h}^0(\tilde{x};\tilde\gamma_0, \tilde\psi^*_0,\tilde{\Phi}) 
    \leq {h}(\hat{x};\hat\gamma, \hat\psi^*,\hat{\Phi}).
\]
Note that ${h}(\hat{x};\hat\gamma, \hat\psi^*,\hat{\Phi})$ is a lower bound on the optimal value of the hNLP-Id relaxation for the subproblem derived from the $\Sfix$ branching strategy corresponding to $x_j=0$. Therefore, $
\znlp(\tilde{C}^0_1,\tilde{C}^0_2,s;\tilde\psi_0^*,\tilde{\Phi})
$
is not greater  than the optimal value of this relaxation.

With an analogous analysis and by considering \eqref{for1}, we conclude that $
\znlp(\tilde{C}^1_1,\tilde{C}^1_2,s-1;\tilde\psi_1^*,\tilde{\Phi}) 
$
is not greater  than the optimal value of the hNLP-Id relaxation for the subproblem derived from the $\Sfix$ branching strategy corresponding to $x_j=1$.
The result follows.
\qed
\end{proof}

\begin{remark}
    Theorem \ref{thm:Rfix_better_NLPId_over_Ffix} does not extend to the \rm{hNLP-Di} and \rm{hNLP-Tr}, whose objective functions are not in general convex or monotone in $\gamma$.
\end{remark}

Theorem \ref{thm:Rfix_better_NLPId_over_Ffix} establish that the $\Rfix$ branching strategy is no worse than the $\Sfix$ branching strategy. The following example demonstrates that the domination can be strict, both when we fix a variable at zero or at one.
\begin{example}\label{ex:RvsC}
    Consider an instance $C:=\left(\begin{smallmatrix}
        C[N,N] & C[N,T]\\
        C[N,T]^\top & C[T,T]
    \end{smallmatrix}\right)$, with $n:=4$, $t :=4$, $s:= 2$, and where  
    $
    C[N,N] := \Diag((0.1,2,2,2)),~ C[N,T] := \Diag((1/\sqrt{110},1,1,1)),~C[T,T] := I_4$\,. Then, for $\Phi:=\mathbf{e}$, we have 
\begin{center}
\begin{tabular}{c|c|cc|c}
        & & \multicolumn{2}{c|}{ \rm{ \NLPcompMERSP}} & \\[-2pt]
        & & \multicolumn{2}{c|}{\rm{bound}} & \\[2pt]
\rm{case} &\rm{strategy}& $\psi:=1$
& $\psi:=\psi^*$& $\psi^*$ \\[2pt] \hline
\multirow{2}{*}{$x_1 = 0$}
  & $\Sfix$ & $-0.24$ & $-0.41$ & $1.1$ \\
  & $\Rfix$ & $-1.22$ & $-1.39$ & $2.0$ \\ \hline
\multirow{2}{*}{$x_1 = 1$}
  & $\Sfix$ & $-0.16$ & $-0.25$ & $1.1$ \\
  & $\Rfix$ & $-0.64$ & $-0.79$ & $2.0$
\end{tabular}
\end{center}
 \hfill $\clubsuit$
\end{example}

\begin{remark}\label{rem:RfixvsFfix}
    Theorem \ref{thm:Rfix_better_NLPId_over_Ffix} and Example \ref{ex:RvsC} illustrate the  superiority of the branching strategy $\Rfix$ over the branching strategy $\Sfix$. In addition, the  subproblems generated by  $\Rfix$ have dimension $n-1$, whereas those generated by  $\Sfix$ retain the original dimension. Consequently, from a computational standpoint, $\Rfix$ is generally expected to be more efficient than $\Sfix$, as it produces stronger relaxations while simultaneously reducing the size of the subproblems.
\end{remark}


\section{Numerical experiments}\label{sec:exp}

We used the general-purpose solver IPOPT
1.14.3 (see \citep*{IPOPT}) 
and the conic solver MOSEK 11.2.0 (see \citep*{MOSEK}), via the Julia wrapper MOSEKTools.jl  v0.15.10) to solve the convex optimization problems. These are commonly used for the kind of problems that we solve; more specifically, we used  
IPOPT
for all relaxations except to  solve the problems \eqref{eq:nlp-trace} and \eqref{eq:min_maxeig_C2_subprob_reformulation}. All of our algorithms were implemented in Julia v1.11.9. 
We used the parameter settings for the solvers aiming at their best performance. For IPOPT, we used:
\texttt{TOL} = $10^{-7}$ (convergence tolerance). 

For hNLP-Di and hNLP-Tr, the scaling factor $\gamma$ should be selected in the interval between $1/d_{\max}$ and $1/d_{\min}$\,. In our experiments, we have tested 50 evenly spaced values for $\gamma$ in this interval and report the results for the best one.  

We ran our experiments on macOS 15.6: an Apple M4 Pro chip with 14 cores (10 performance and 4 efficiency) and 48 GB of memory. 

When comparing different upper bounds for our test instances, we present the gaps defined by the difference between the upper bounds for \ref{MERSP}, and the lower bounds computed by a local-search heuristic.

We use labels 
strategy, strategy$^{\psi}$, strategy$^{\psi}_\Phi$, to indicate that the hyper-scaled NLP bound was computed using the parameter-selection rule \emph{strategy} $\in\{\IdentityNLP, \DiagonalNLP, \TraceNLP\}$, with the following choices of scaling parameters:
\begin{itemize}
    \item \emph{strategy} : $\psi:=1$ and $\Phi:=\mathbf{e}$;
    \item \emph{strategy}$^{\psi}$ : $\psi:=\psi^*$ and $\Phi:=\mathbf{e}$;
    \item \emph{strategy}$^{\psi}_\Phi$ : $\psi:=\psi^*$ and $\Phi:=\hat{\Phi}$,
\end{itemize}
where $\hat{\Phi}$ is obtained by solving  \eqref{eq:min_maxeig_C2_subprob}. For the $\IdentityNLP$\ strategy, the BFGS procedure described in Section \ref{subsec:optimize-Psi} is additionally applied to further improve the value of $\Phi$. We note that we compute $\psi^*$ considering the g-scaled matrices.

In \eqref{eq:min_maxeig_C2_subprob}, the matrix $\Diag(\Phi)C_2\Diag(\Phi)$ depends quadratically on $\Phi$, so the semidefinite constraint is not affine in the decision variables. Consequently, a standard semidefinite-programming solver requires a  Schur-complement lifting to an LMI of size $2n\times 2n$. As $C_2\succeq0$,  the constraint $\tau I_n-\Diag(\Phi)C_2\Diag(\Phi)\succeq 0$ is equivalent to
\[
\begin{pmatrix}\tau I_n & \Diag(\Phi)C_2^{1/2}\\ C_2^{1/2}\Diag(\Phi) & I_n\end{pmatrix}\succeq 0,
\]
which is affine in $(\tau,\Phi)$. 

The following result eliminates this lifting. By exploiting a factorization $C_2=FF^\top$, it yields an equivalent convex formulation for \eqref{eq:min_maxeig_C2_subprob}, whose semidefinite constraint is  affine in the decision variables and has size $k\times k$, where $k$ may be chosen equal to  $\rank(C_2)$\,. We use this formulation to compute $\Phi$ in our computational experiments.

\begin{proposition}
    Let $C_2 \succeq 0$ and consider a factorization $C_2 = FF^\top$, with $F \in \mathbb{R}^{n \times k}$ for some $k$ satisfying $\rank(C_2) \leq k \leq n$. Then \eqref{eq:min_maxeig_C2_subprob} is equivalent to the following convex optimization problem:
    \begin{equation}\label{eq:min_maxeig_C2_subprob_reformulation}
         \min\{\tau \,:\,\tau I_k - F^\top\Diag(w)F\succeq 0,\,\left(\,{\textstyle\prod_{i=1}^n w_{i}}
\,\right)^{1/n}\geq   1,\, w \geq \mathbf{0}\}.
    \end{equation}
\end{proposition}
\begin{proof}
    To prove the equivalence, we demonstrate that any feasible solution of one problem yields a feasible solution of the other with the same objective value. 
    \begin{itemize}
        \item[(i)] Let $(\hat\tau,\hat\Phi)$ be a feasible solution of \eqref{eq:min_maxeig_C2_subprob}. Note that $\hat \Phi \in \mathbb{R}^n_{++}$\,.
        Define $\hat{w}_i := \hat\Phi_i^2$ for $i \in N$. The strict positivity $\hat\Phi_i > 0$ gives $\hat w_i > 0$. Moreover, as  $\prod_{i=1}^n \hat\Phi_i \geq 1$, we have that
        \[\textstyle
\big(\prod_{i=1}^n \hat w_i\big)^{1/n}
= \big(\prod_{i=1}^n \hat\Phi_i^2\big)^{1/n}
= \big(\prod_{i=1}^n \hat\Phi_i\big)^{2/n}
\geq 1.
\]
Let  $\hat M := \Diag(\hat\Phi)F$. Then
\begin{equation}\label{eq:prop_reformulation_C2eigmin_M}
    \Diag(\hat\Phi) C_2 \Diag(\hat\Phi) = \Diag(\hat\Phi) F F^\top \Diag(\hat\Phi) = \hat M \hat M^\top,
~~~ 
F^\top \Diag(\hat w) F = F^\top \Diag(\hat\Phi)^2 F = \hat M^\top \hat M.
\end{equation}
The matrices $\hat M \hat M^\top$ and $\hat M^\top \hat M$ share the same nonzero eigenvalues, so $\lambda_{1}(\hat M \hat M^\top) = \lambda_{1}(\hat M^\top \hat M)$, and therefore
\[
\Diag(\hat\Phi) C_2 \Diag(\hat\Phi) \preceq \hat\tau I_n ~\Rightarrow~F^\top \Diag(\hat w) F \preceq \hat\tau I_k\,.
\]
Hence, $(\hat\tau, \hat w)$ is feasible for \eqref{eq:min_maxeig_C2_subprob_reformulation}.

\item[(ii)] Let $(\hat\tau,\hat w)$ be a feasible solution of \eqref{eq:min_maxeig_C2_subprob_reformulation}. Note that $\hat w \in \mathbb{R}^n_{++}$\,.
Define $\hat{\Phi}_i := \sqrt{\hat w_i}$ for $i \in N$. The strict positivity $\hat w_i > 0$ gives $\hat \Phi_i > 0$. Moreover, $\prod_{i=1}^n \hat w_i \geq 1$ implies
        \[\textstyle
\big(\prod_{i=1}^n \hat \Phi_i\big)^{1/n}
= \big(\prod_{i=1}^n \sqrt{\hat w_i}\big)^{1/n}
= \big(\prod_{i=1}^n \hat w_i\big)^{1/2n}
\geq 1.
\]
Using again \eqref{eq:prop_reformulation_C2eigmin_M}, 
$\lambda_1(F^\top \Diag(\hat w) F) = \lambda_1(\Diag(\hat\Phi) C_2 \Diag(\hat\Phi))$, and therefore
\[
F^\top \Diag(\hat w) F \preceq \hat\tau I_k ~\Rightarrow~\Diag(\hat\Phi) C_2 \Diag(\hat\Phi) \preceq \hat\tau I_n\,.
\]
Hence, $(\hat\tau, \hat \Phi)$ is feasible for \eqref{eq:min_maxeig_C2_subprob}.     \qed
    \end{itemize}
\end{proof}


\subsection{Dominance of the complementary NLP bound over the spectral bound for \ref{MERSP}}\label{subsec:dom-nlp-bound}

In Theorem \ref{thm:nlp_dom_spec}, we presented sufficient conditions under which the complementary NLP bound dominates the spectral bound for \ref{MERSP}. In the following experiment, we investigate whether the weakest condition is commonly satisfied in our test instances and how critical its satisfaction is for the dominance to occur. 

For our experiments, we consider three positive-definite covariance matrices of orders $63,\,90,\,124$
 that have been extensively used in the \ref{MESP} literature,  (see, e.g., \citep*{KLQ}, \citep*{LeeConstrained}, \citep*{AFLW_Using}, \citep*{Anstreicher_BQP_entropy}, \citep*{Kurt_linx}, \citep*{gscale}, \citep*{li2025augmented}, \citep*{ADMM4DOPT}, \citep*{MESP2DOPT}, \citep*{ExtendedGMESP}). In the figures that follows, instances corresponding to these matrices are denoted  by $\mathbf{C}(63)$, $\mathbf{C}(90)$, $\mathbf{C}(124)$ respectively. We set $n=40$ and vary $t$. The covariance matrices $C$ used in our instances are constructed as the leading principal submatrices of order $n+t$ from the corresponding benchmark matrices.
 
 In Figure \ref{fig:compare_sc1}, we present six plots arranged in two rows and three columns. In the first row, we give the values of $\Delta:=\textstyle\sum_{\ell=1}^{n-s}\log \lambda_{n-\ell+1}( D_{
        \hat\eta} B D_{\hat\eta})  - \log \lambda_{\ell}( D_{\hat\eta} B_T D_{\hat\eta} )$ where $\hat\eta := \argmin v(\eta)$
        and we highlight zero with a black line. For all the values of $s$ and $t$ such that $\Delta$ is below this black line, Theorem \ref{thm:nlp_dom_spec} guarantees that  the complementary NLP bound dominates the spectral bound. In the second row, we give the actual difference between the complementary NLP bound and the spectral bound (``Bound difference'') given by 
        \[
        \znlp(B^{-1},B_T^{-1},n-s;\psi\!:=\!1,\Phi\!:=\!\mathbf{e}) + \ldet(B) - \ldet(B_T) -\zspectral(B,B_T\,,s),
        \]
        where  the complementary NLP bound considered   is the  minimum upper bound obtained with the three strategies (Identity, Diagonal, and Trace)  for selecting parameters in \ref{hNLP}.
        We again include a black line at zero.

\begin{figure}[!ht]
    \centering    
    \includegraphics[width=\linewidth]{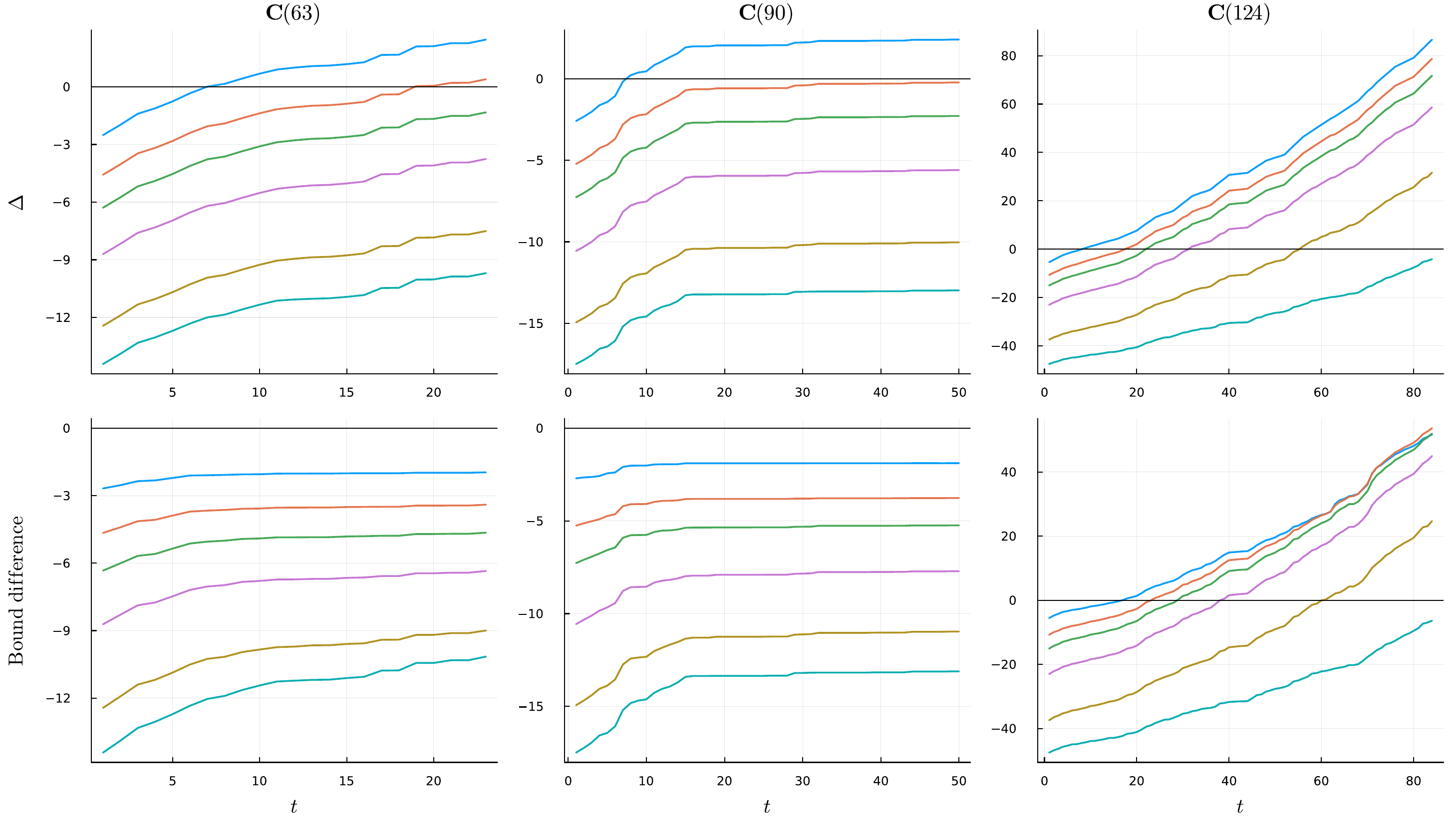}

    \includegraphics[width=0.6\linewidth]{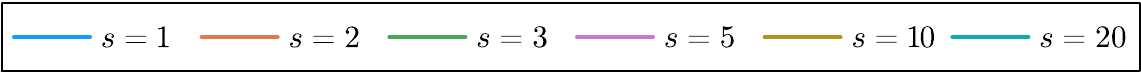}
    \caption{Dominance  condition ($\Delta\!\leq\! 0$) for  complementary NLP bound for \ref{MERSP} over spectral bound for \ref{MERSP}, $n=40$}
    \label{fig:compare_sc1}
\end{figure}
We observe that for $\mathbf{C}(63)$ and $\mathbf{C}(90)$, the weakest condition in Theorem \ref{thm:nlp_dom_spec} ($\Delta\!\leq\! 0$),   is  satisfied for all $t$ when $s\geq 3$, and that the complementary NLP bound strictly dominates the spectral bound for all tested values of $t$ and $s$, even when $\Delta \!> \!0$. 
For $\mathbf{C}(124)$, we note more significant violations of the condition, and in several of such cases, the spectral bound outperforms the complementary NLP bound. Finally, we note that even after violating the strongest condition $t\! \leq\! s$, the inequality  $\Delta \!\leq\! 0$ still holds for some values of $t$. Similarly, after  violating $\Delta \!\leq\! 0$,  the dominance of the complementary NLP bound may still persist. However, for all $s \!\leq\! 10$ and sufficient large $t$ relative to $s$,  $\Delta \!>\!0$ and the spectral bound dominates the complementary NLP bound. We conclude that the proved sufficient condition for the dominance of the NLP bound over the spectral bound is not necessary, but its large violation suggests better performance of the spectral bound.


\subsection{Impact of the proposed scaling procedures on the complementary NLP bound}\label{subsec:res-aug-procedures}

In \S\S\ref{sec:aug-nlp-bound}--\ref{sec:diag-scale}, we discussed strategies for computing the scaling parameters $\psi$ and $\Phi$ in \ref{hNLP} with the goal of  obtaining the strongest possible bound for \ref{MEDP}. We now evaluate the impact of these scaling procedures on the complementary NLP bound for \ref{MERSP} proposed in \citep*{AFLW_Remote}, using the three alternatives to select the parameters $\gamma$ and $d$: $\IdentityNLP$, $\DiagonalNLP$, and $\TraceNLP$. The experiments use the same instances as those considered in \S\ref{subsec:dom-nlp-bound}, all with $n=40$. Now, we set $s:=20$. In all tests, the scaling parameter $\psi^*$ was selected according to Theorem \ref{thm:best_alpha_NLPId},  
and the  g-scaling parameter $\Phi$ was  obtained by solving  \eqref{eq:min_maxeig_C2_subprob}. For the $\IdentityNLP$\ strategy, the BFGS procedure described in \S\ref{subsec:optimize-Psi} was additionally applied to further improve the value of $\Phi$. 

In Figure \ref{fig:compare_nlp_procedures_full_rank_comp_bounds_C63_90_124}, we give the gaps obtained after applying both scaling procedures to the complementary NLP bound.
The $\IdentityNLP$\ strategy provides the best results. This is a particularly appealing outcome because the choice of $\gamma$ is uniquely determined for hNLP-Id, whereas the alternative parameter-selection strategies 
 require substantial additional computational effort to identify  a value of $\gamma$  that produces a high-quality bound.  
  For $\mathbf{C}(124)$, the $\TraceNLP$\ strategy  performs almost as well as the $\IdentityNLP$\ strategy. This behavior is consistent with Theorem \ref{thm:nlp_id_eq_nlp_tr}, which demonstrates that when $\Phi$ is given by the solution of  \eqref{eq:min_maxeig_C2_subprob}, the parameters choice for the $\TraceNLP$\ strategy coincides with those for the $\IdentityNLP$\ strategy. The small improvement of $\IdentityNLP^\psi_\Phi$ over $\TraceNLP^\psi_\Phi$ is due to  the procedure described in \S\ref{subsec:optimize-Psi}, which further optimizes the g-scaling parameter $\Phi$ for $\IdentityNLP^\psi_\Phi$\,.  
 This indicates that, for $\mathbf{C}(124)$, optimizing over $\Phi$ leads only to a marginal improvement which is, in fact,  not visible at the scale of the plot. Finally, we note that the $\DiagonalNLP$ parameter-selection strategy provides the worst results among the three options.

\begin{figure}[!ht]
    \centering
    \includegraphics[width=0.99\linewidth]{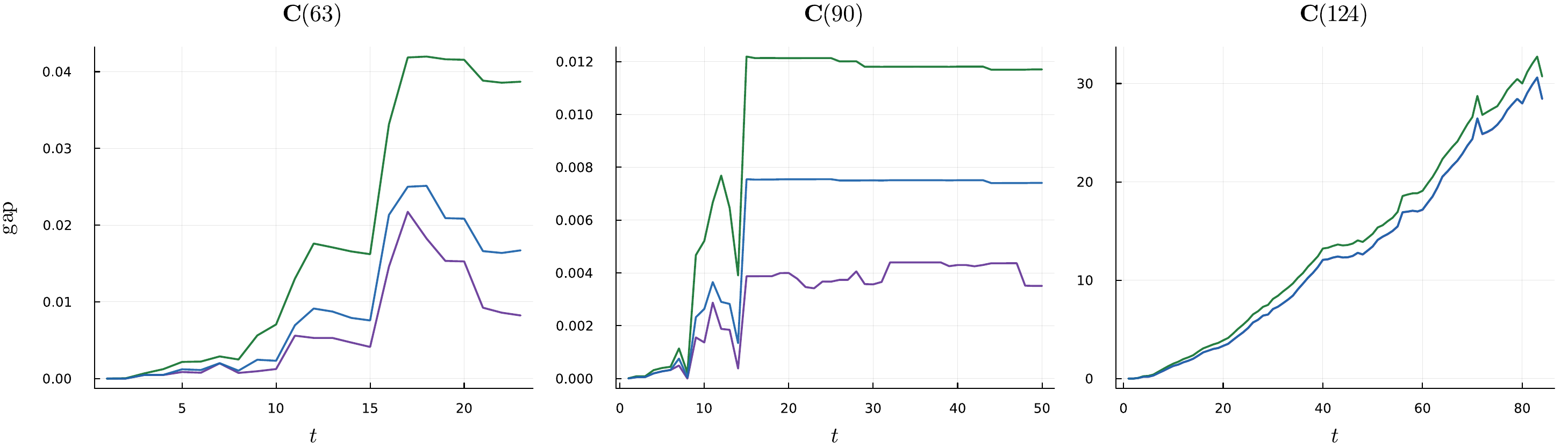}
    \includegraphics[width=0.5\linewidth]{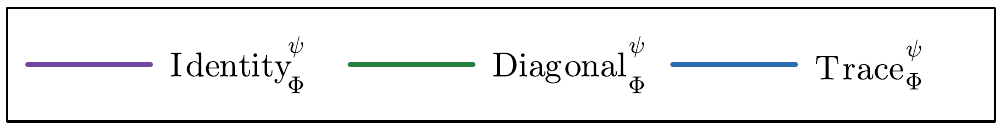}
    \caption{Gaps for the \NLPcompMERSP\ bound, $n=40$, $s=20$
    }
    \label{fig:compare_nlp_procedures_full_rank_comp_bounds_C63_90_124}
\end{figure}

In Figure \ref{fig:compare_nlp_improvement_procedures_full_rank_comp_bounds_C63_90_124}, we illustrate the impact of each scaling procedure when applied to  the complementary NLP bound. We focus on the $\IdentityNLP$ strategy, as it provided  the best results in Figure \ref{fig:compare_nlp_procedures_full_rank_comp_bounds_C63_90_124}.
    We first analyze the effect of the scaling-parameter $\psi$. For $\mathbf{C}(63)$ and $\mathbf{C}(90)$, Identity$^\psi$ coincides with  Identity. In contrast, for $\mathbf{C}(124)$, we see a noticeable improvement using $\psi=\psi^*$ for $t \geq 40$; this is  expected: Corollary \ref{cor:CNT_full_row_rank} establishes that the condition $\rank(C[N,T]) = n$ is necessary for the \NLPcompMERSP\ bound  to  strictly dominate the complementary NLP bound, and this condition is satisfied only for $\mathbf{C}(124)$. We then analyze the impact of g-scaling and observe a significant reductions in the gaps.  For $\mathbf{C}(63)$ and $\mathbf{C}(90)$, most gaps are reduced to values close to zero. Even for the most challenging instance, $\mathbf{C}(124)$, g-scaling provides a significant improvement: for $t=84$, the gap of $\IdentityNLP_{\Phi}$\ is approximately $32\%$ smaller than that of  $\IdentityNLP$. It is also interesting to note that both $\Phi$ and $\psi$ contribute to the enhancement of the bound, and neither alone dominates the other, as illustrated by the results for  $\mathbf{C}(124)$.

\begin{figure}[!ht]
    \centering
    \includegraphics[width=0.99\linewidth]{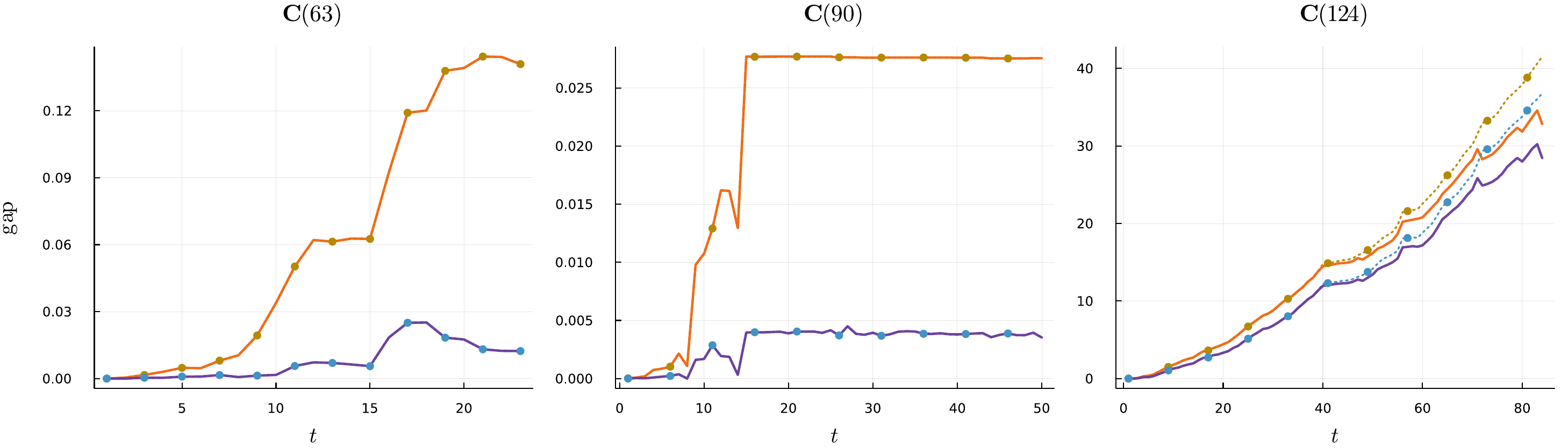}
    \includegraphics[width=0.6\linewidth]{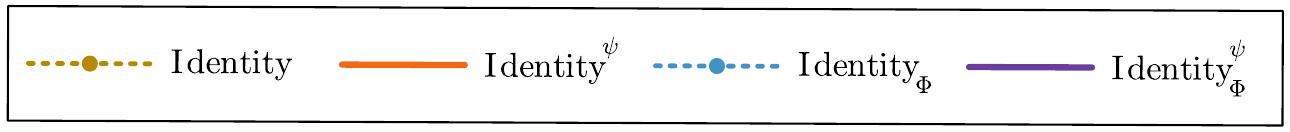}
    \caption{Impact of scaling procedures on the complementary NLP-Id bound, $n=40$, $s=20$ 
}\label{fig:compare_nlp_improvement_procedures_full_rank_comp_bounds_C63_90_124}
\end{figure}


\subsection{MERSP instances with singular covariance matrix}\label{subsec:exp-singular-case}

In \S\ref{subsec:singular-case-thms}, we demonstrated that the scaling parameter $\psi$ extends the application of the \ref{hNLP} bound to a broader class of \ref{MERSP} instances, including those with singular covariance matrices, provided that condition \eqref{eq:cond-psi-geq0-exists} holds. For such  singular instances, the \ref{hNLP} bound, whether used with or without g-scaling, is currently the only available upper-bounding method.  

For the following experiment, we considered the rank-deficient matrix of order 2000 with rank 949 based on Reddit data from  \citep*{Dey2018} and \citep*{Munmun}. This covariance matrix has also been used in experiments discussed in  
\citep*{Weijun,FactPaper,ADMM4DOPT,MESP2DOPT}.  From the order-2000 matrix, we constructed two singular covariance matrices $C$ with distinct features. The main difference between the approaches used to construct the matrices lies in how the sets $N$ and $T$ are chosen from rows (and columns) of the order-2000 matrix. 

For the first matrix constructed ($C^a$), $N$ and $T$ are selected solely to satisfy the two required properties $C[T,T]\succ 0$ and \eqref{eq:cond-psi-geq0-exists}, ensuring that the problem is well-defined. For the second matrix ($C^b$), $N$ and $T$ are chosen not only to satisfy these requirements but also with the goal of maximizing the optimal value of the scaling parameter $\psi^*$ (see \eqref{psistar_orig}). Considering Lemma \ref{lem:nlp_aug_psi_decrease}, this selection aims to accentuate the effect of $\psi$-scaling. 

In Figures \ref{fig:singular-aug-nlp-strategies}--\ref{fig:good_singular_instance_nlp}, we give results for instances with the covariance matrices $C^a$ and $C^b$, with $n=50$, $s=25$ and varying values of $t$. 
In Figure \ref{fig:singular-aug-nlp-strategies}, we present the best final gap obtained using both scaling procedures. Consistent with the behavior observed for \NLPcompMERSP\ in Figure \ref{fig:compare_nlp_procedures_full_rank_comp_bounds_C63_90_124}, the $\IdentityNLP$ strategy yields the strongest bounds for \NLPorigMERSP, with $\IdentityNLP^\psi_\Phi$ performing slightly better than $\TraceNLP^\psi_\Phi$, whereas $\DiagonalNLP^\psi_\Phi$ exhibits the weakest performance. An analogous observation concerning the similarity of the $\IdentityNLP^\psi_\Phi$ and $\TraceNLP^\psi_\Phi$ bounds for $\mathbf{C}(124)$ in Figure \ref{fig:compare_nlp_procedures_full_rank_comp_bounds_C63_90_124} applies here as well. We also observe that the approach used to construct $C^a$ and $C^b$ has a significant impact on the quality of the bounds. For $C^b$, the bounds are very tight, and for $C^a$ with large $t$, the gaps are much larger. 

\begin{figure}[!ht]
    \centering
\includegraphics[width=0.99\linewidth]{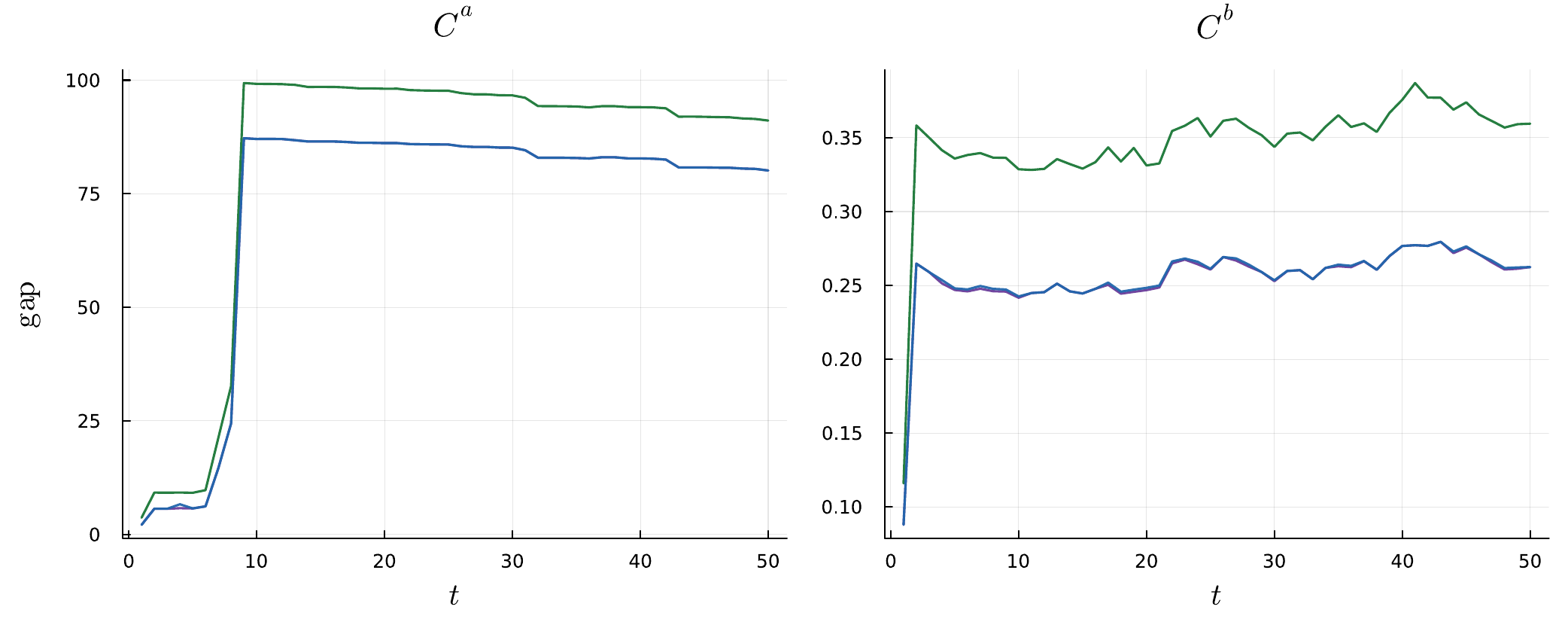}

\includegraphics[width=0.6\linewidth]{best_gaps_singular_legend.pdf}
    \caption{Gaps for the \NLPorigMERSP\ bound on \ref{MERSP} instances with singular covariance matrices, $n\!=\!50,s\!=\!25$
    }
\label{fig:singular-aug-nlp-strategies}
\end{figure}

In Figure \ref{fig:values_of_psi_singular_instances_nlp}, we  see that the gaps presented in Figure \ref{fig:singular-aug-nlp-strategies} are inversely proportional to the values of $\psi^*$: values of $\psi^*$ close to one correspond to better performance of the \NLPorigMERSP\ bound. For $C^a$, sharp decreases in $\psi^*$, as revealed on the logarithmic scale, are associated with noticeable increases in the gaps. In contrast, for $C^b$, all instances exhibits values $\psi^*\geq 0.95$, which correspond to small gaps. 
\begin{figure}[!ht]
    \centering
    \includegraphics[width=0.6\linewidth]{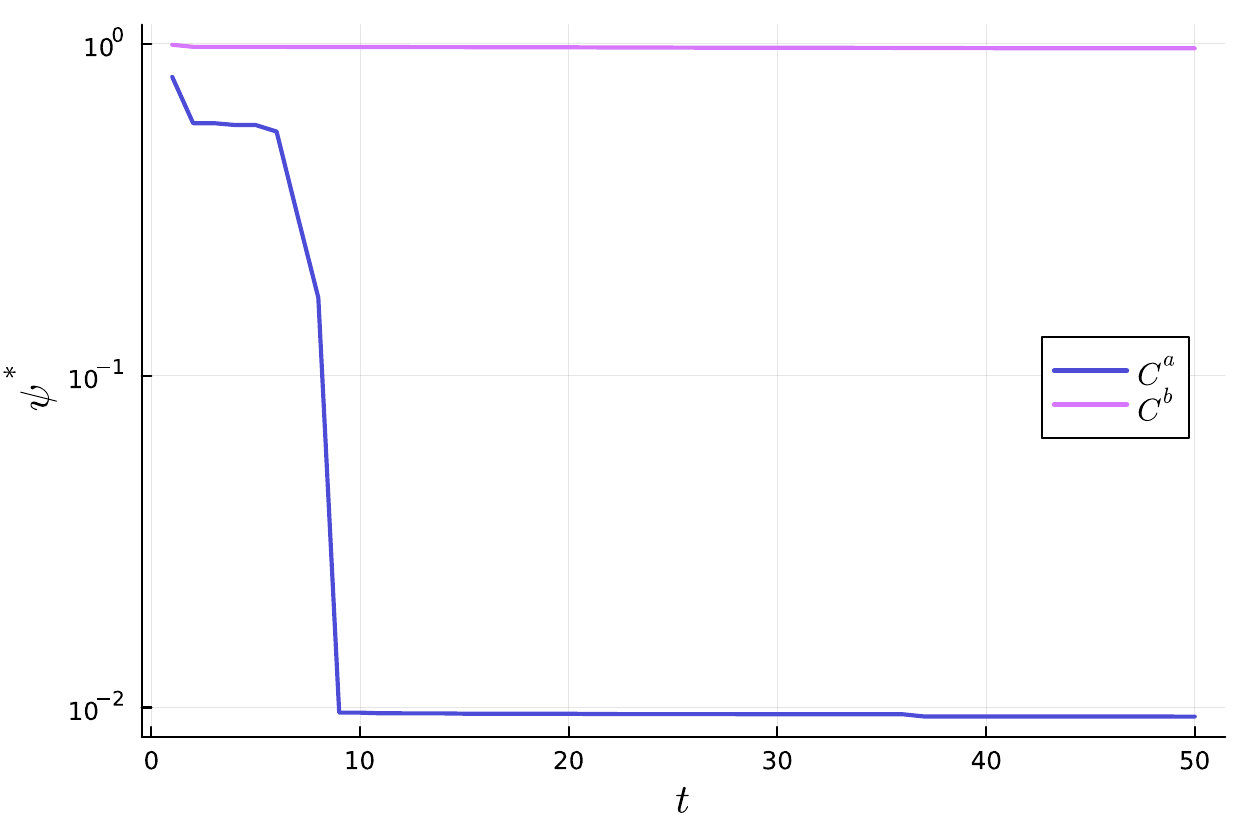}
    \caption{Values of $\psi^*$ in \NLPorigMERSP\ for the \ref{MERSP} instances  considered in Figure \ref{fig:singular-aug-nlp-strategies}
\label{fig:values_of_psi_singular_instances_nlp}}
\end{figure}

 Finally,  Figure \ref{fig:good_singular_instance_nlp}, illustrates the effect of g-scaling on the \NLPorigMERSP\ bound, focusing on the $\IdentityNLP$  strategy, for which  the g-scaling parameter is optimized as described in \S\ref{subsec:optimize-Psi}. The behavior observed for singular covariance matrices closely mirrors the reported in  Figure \ref{fig:compare_nlp_procedures_full_rank_comp_bounds_C63_90_124}  for the \NLPcompMERSP\ bound in the nonsingular case. Again, we observe a substantial reduction in the gaps when g-scaling is applied. 

\begin{figure}[!ht]
    \centering
\includegraphics[width=0.99\linewidth]{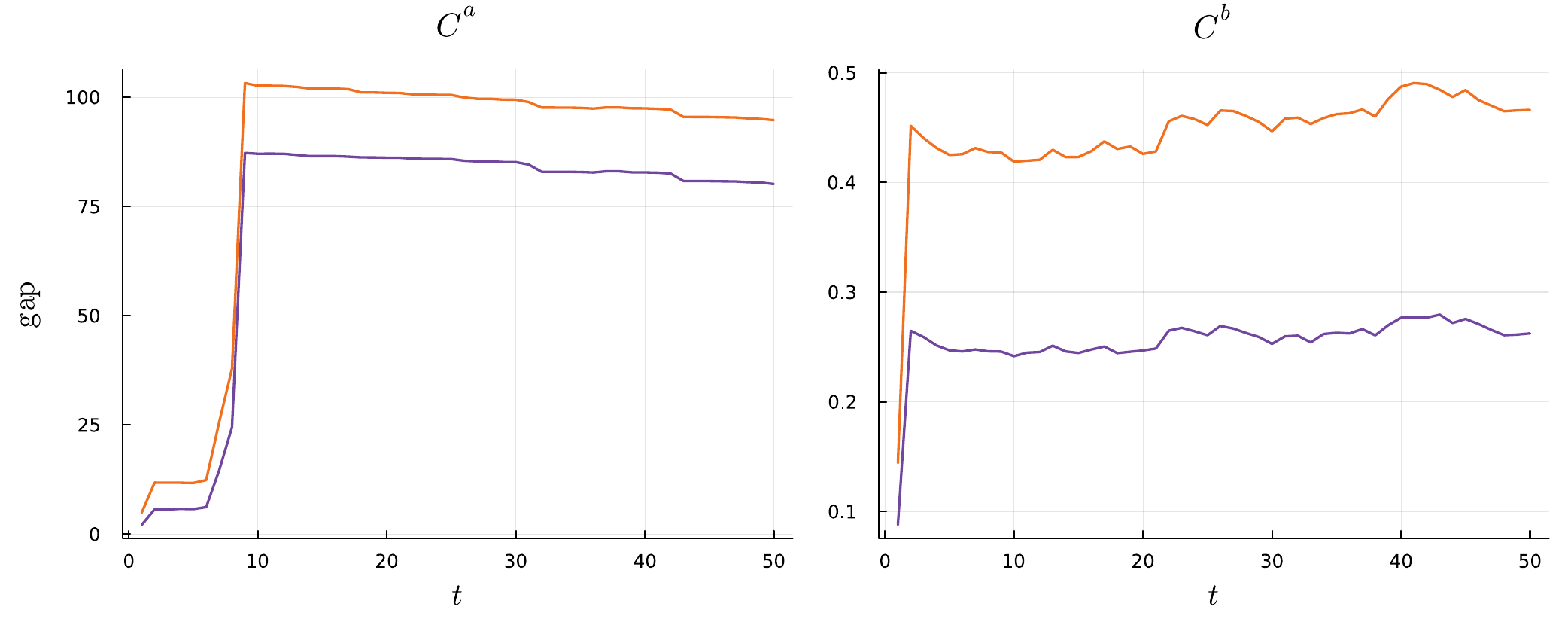}

\includegraphics[width=0.3\linewidth]{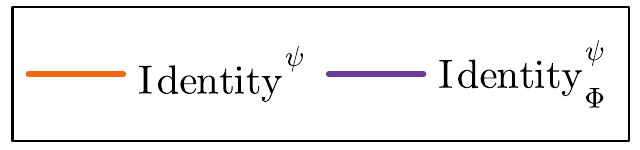}
     \caption{Impact of g-scaling on the \NLPorigMERSP\ bound for the \ref{MERSP} instances considered in Figure \ref{fig:singular-aug-nlp-strategies}}  
\label{fig:good_singular_instance_nlp}
\end{figure}

\subsection{Branch-and-bound for \ref{MERSP}}
We implemented a branch-and-bound (B\&B) algorithm  using the Julia package \texttt{Bonobo.jl}
(see \url{https://github.com/Wikunia/Bonobo.jl/})
with a pseudo-cost branching strategy. At each node, variable fixing was performed through the $\Rfix$ strategy, which reduces the size of the data matrix and is expected to outperform $\Sfix$ both theoretically and computationally (see Remark \ref{rem:RfixvsFfix}). We also applied the additional variable-fixing procedure from Theorem \ref{thm:variable_fixing}.

Our computational study focuses on  $\mathbf{C}(124)$, the most challenging of the three non-singular test instances, with $n=40$, $s=20$, and $t = 5, \ldots, 10$. Even for these relatively small values of $t$, the root-node gaps remain substantial, making the instances non-trivial for exact solution. Consequently, a time limit of two hours was imposed for each instance. Whenever this 
limit was reached, the execution time is reported as `*' in the presentation of our numerical results. 

The upper-bounding procedures were  $\IdentityNLP^\psi$\ and $\IdentityNLP^\psi_{\Phi}$\,. This choice was motivated by the computational results of Figure \ref{fig:compare_nlp_procedures_full_rank_comp_bounds_C63_90_124} and Figure \ref{fig:compare_nlp_improvement_procedures_full_rank_comp_bounds_C63_90_124}, which show that the $\IdentityNLP$\ strategy provides the strongest bounds and both scaling mechanisms contribute to tightening the relaxation. The parameter $\psi$ was always employed, as it admits a closed-form solution  (Theorem \ref{thm:best_alpha_NLPId}) and can be computed at negligible cost throughout the B\&B tree.

To evaluate $\IdentityNLP^\psi_{\Phi}$ at each node, we computed a high-quality $\Phi$ by solving  \eqref{eq:min_maxeig_C2_subprob}. In our experiment, further optimizing $\Phi$ using the procedure described in \S\ref{subsec:optimize-Psi} produced only  marginal improvements when initialized from the solution of \eqref{eq:min_maxeig_C2_subprob}, typically below 0.01 in the resulting bound, while  substantially increasing computational effort. For this reason, only the  scaling obtained from \eqref{eq:min_maxeig_C2_subprob} was employed  within the B\&B algorithm.  

Table \ref{tab:bb_nlp_comparison} summarizes the results. The effect of g-scaling is already evident at the root node, where the gaps are reduced by approximately 27\% to 44\% across all tested values of $t$. These stronger bounds translate directly into a much smaller B\&B search tree. For example, when $t=9$, the number of explored nodes decreases from 386755 to 107329, while the running time drops from the two-hour limit to approximately 4600 seconds. Similar reductions are observed for all values of $t$. 
The impact on solution quality is even more striking. For $t=8$ and $t=9$, $\IdentityNLP^\psi_{\Phi}$\ closes the gap completely within the time limit, whereas $\IdentityNLP^\psi$\ still leaves gaps of 0.19 and 0.38, respectively.  Even for the most difficult instance ($t=10$), g-scaling reduces the final gap from 0.60 to 0.07. 

An interesting observation is that the contribution of $\psi$-scaling inside B\&B is relatively limited for these instances. At the root node, $\psi^*=1$, so the $\IdentityNLP^\psi$\ bound coincides with the original $\IdentityNLP$\ bound. This is expected from Corollary \ref{cor:CNT_full_row_rank} because $t<n$. As variables are fixed and the reduced subproblems become smaller, values $\psi^*>1$ may arise. However, the last columns of Table \ref{tab:bb_nlp_comparison} show that such nodes represent only a small fraction of the search tree. For instance, when $t=10$, only  472 of the 174075 nodes explored by $\IdentityNLP^\psi_\Phi$\ have $\psi^*>1$. Thus, for these instances, the main benefit of the $\psi$-scaling lies in its negligible computational cost and its ability to strengthen the relatively small number of subproblems reported in the last two columns of Table \ref{tab:bb_nlp_comparison}, whereas the dominant improvements in the B\&B performance were obtained through g-scaling.   

Finally, the variable-fixing procedure proves effective, particularly for $\IdentityNLP^\psi_\Phi$. The number of fixed variables grows substantially with $t$, exceeding $10^5$  when $t=9$ and $t=10$. The results also highlight the strong connection between bound quality and variable fixing: the stronger bounds provided by g-scaling lead to significantly more effective variable fixing.  In particular, for  $t=9$ and $t=10$,  $\IdentityNLP^\psi_\Phi$\ fixes substantially more variables than  $\IdentityNLP^\psi$\, despite exploring far fewer nodes. For smaller values of $t$, $\IdentityNLP^\psi$\, may  fix more variables overall, but these instances are solved by $\IdentityNLP^\psi_\Phi$ after exploring relatively few nodes, limiting the potential impact of additional variable fixing. 

These results suggest that stronger upper bounds not only improve node pruning directly but also enhance the effectiveness of the variable-fixing procedure. The resulting reduction in subproblems dimensions is therefore an important factor behind the overall performance of the B\&B algorithm and complements the improvements provided by the tighter relaxations. 

\begin{table}[!ht]
\setlength{\tabcolsep}{5.5pt}
\scriptsize
\centering
\begin{tabular}{r|cc|cc|rr|rr|rr|rr}
 & \multicolumn{2}{c|}{\scriptsize root gap} & \multicolumn{2}{c|}{\scriptsize final gap} & \multicolumn{2}{c|}{\scriptsize time (sec)} & \multicolumn{2}{c|}{
 \# nodes} & \multicolumn{2}{c|}{
 \# fixed variables} & \multicolumn{2}{c}{
 \# nodes with $\psi^*\!>\!1$} \\[1.75pt]
\multicolumn{1}{c|}{\scriptsize $t$} & \multicolumn{1}{c}{$\IdentityNLP^\psi$} & \multicolumn{1}{c|}{$\IdentityNLP^\psi_\Phi$} & \multicolumn{1}{c}{$\IdentityNLP^\psi$} & \multicolumn{1}{c|}{$\IdentityNLP^\psi_\Phi$} & \multicolumn{1}{c}{$\IdentityNLP^\psi$} & \multicolumn{1}{c|}{$\IdentityNLP^\psi_\Phi$} & \multicolumn{1}{c}{$\IdentityNLP^\psi$} & \multicolumn{1}{c|}{$\IdentityNLP^\psi_\Phi$} & \multicolumn{1}{c}{$\IdentityNLP^\psi$} & \multicolumn{1}{c|}{$\IdentityNLP^\psi_\Phi$} & \multicolumn{1}{c}{$\IdentityNLP^\psi$} & \multicolumn{1}{c}{$\IdentityNLP^\psi_\Phi$} \\[1.75pt] \hline
5{$\vphantom{\Sigma^{I^I}}$} & 0.36 & 0.20 & 0.00 & 0.00 & 471.06 & 45.32 & 24577 & 659 & 12494 & 1481 & 3 & 2 \\
6 & 0.52 & 0.31 & 0.00 & 0.00 & 1283.10 & 100.28 & 68297 & 1747 & 26961 & 3065 & 17 & 5 \\
7 & 0.85 & 0.56 & 0.00 & 0.00 & 5254.23 & 330.84 & 242331 & 7927 & 89746 & 9935 & 11 & 45 \\
8 & 1.19 & 0.80 & 0.19 & 0.00 & * & 1261.06 & 398051 & 26597 & 57337 & 22333 & 0 & 16 \\
9 & 1.50 & 1.05 & 0.38 & 0.00 & * & 4624.15 & 386755 & 107329 & 6638 & 104207 & 0 & 223 \\
10 & 1.77 & 1.29 & 0.60 & 0.07 & * & * & 436781 & 174075 & 622 & 164371 & 1 & 472
\end{tabular}
\captionof{table}[]{B\&B for \ref{MERSP} using the \NLPcompMERSP\, bound  ($\mathbf{C}(124)$,    $n=40$, $s=20$)
}\label{tab:bb_nlp_comparison}
\end{table}


\section{Outlook}\label{sec:out}
We have considered only g-scaling using the same $\Phi$ for both matrices in order to maintain convexity of the \ref{hNLP} bound. We plan to investigate the use of  different g-scalings and develop strategies for preserving convexity in such a setting.
It is natural to try to
extend other convex-programming bounds from 
\ref{MESP} to 
\ref{MERSP}, namely the ``linx'', ``BQP'' and ``factorization'' bounds.  
For the complementary NLP bound, which was the point of departure for
what we presented, 
\citep*{AFLW_Remote} exploited the fact that 
$C[N,N]\succeq C_T[N,N]$,
 and so $C[N,N]^{-1} \preceq C_T[N,N]^{-1}$.
 This led to a convex-programming bound
 based on a complementary formulation, 
 a shortcoming of which is that it only applies when the covariance matrix $C$ is positive definite --- we partially overcame that restriction.
 Following a similar approach
 with the linx bound and 
 the BQP bound does not seem to
 lead to convex relaxations, and 
 trying to follow such an approach with the factorization bound looks to be very difficult to analyze. So we leave 
this direction for future work. 

\medskip

\noindent {\bf Acknowledgments.} M. Fampa was supported in part by CNPq grant 307167/2022-4.
J. Lee was supported in part by AFOSR grant FA9550-22-1-0172. G. Ponte would like to 
thank Wendel Melo for helpful conversations
on B\&B implementation issues.

\bibliographystyle{cas-model2-names}
\bibliography{Bib}


\section*{Appendix}

\begin{lemma}\label{lem:help_psi_eigvals}
    Let $A\in\mathbb{S}^n_{+}$\,,  $B\in\mathbb{S}^n_{++}$\,, $d\in\mathbb{R}^n_{++}$\,, and $D:=\Diag(d)$.    Let
$Q:=DAD$ and $K:=DBD$.
 We have that
\[
B^{-1/2}AB^{-1/2}
\quad\text{and}\quad
K^{-1/2}QK^{-1/2}
\]
have the same eigenvalues.
\end{lemma}

\begin{proof}
First observe that $K^{-1}=D^{-1}B^{-1}D^{-1}$.
Hence,
\[
K^{-1}Q
=
D^{-1}B^{-1}D^{-1}(DAD)
=
D^{-1}B^{-1}AD.
\]
Now let $S:=D^{-1}$. Then
\[
S^{-1}(K^{-1}Q)S
=
D(D^{-1}B^{-1}AD)D^{-1}
=
B^{-1}A.
\]
Therefore,
\[
K^{-1}Q \sim B^{-1}A,
\]
and thus these matrices have the same eigenvalues.

Next, recall that for any  $E$ positive semidefinite and $F$ positive definite, we have
\[
F^{-1}E
\sim
F^{-1/2}EF^{-1/2}.
\]
Applying this identity with $(E,F)=(A,B)$ and  $(E,F)=(Q,K)$, 
and considering that similar matrices have the same eigenvalues, the results follows.\qed
\end{proof}

\end{document}